\documentclass[10pt,reqno]{amsart}

\usepackage[T1]{fontenc}
\usepackage[utf8]{inputenc}
\usepackage{amsmath,amssymb,mathtools}
\IfFileExists{newtxtext.sty}{\usepackage{newtxtext,newtxmath}}{\usepackage{lmodern}}
\usepackage[a4paper,margin=1.15in]{geometry}
\usepackage{microtype}
\usepackage{xcolor}
\usepackage[colorlinks=true,linkcolor=blue!55!black,citecolor=blue!55!black,urlcolor=blue!55!black]{hyperref}
\allowdisplaybreaks
\numberwithin{equation}{section}
\newcommand{\R}{\mathbb R}
\newcommand{\glyg}{\mathfrak g}
\newcommand{\nalg}{\mathfrak n}
\newcommand{\one}{\mathbf 1}
\newcommand{\RicB}{\Ric^{B}}
\newcommand{\CEadj}{d_{\mu}^{\dagger}}
\newcommand{\geomcod}{d_{g,\mu}^{*}}
\DeclareMathOperator{\Ric}{Ric}
\DeclareMathOperator{\Der}{Der}
\DeclareMathOperator{\ad}{ad}
\DeclareMathOperator{\Sym}{Sym}
\DeclareMathOperator{\tr}{tr}
\DeclareMathOperator{\Id}{Id}

\DeclareMathOperator{\im}{im}
\DeclareMathOperator{\diag}{diag}
\DeclareMathOperator{\Skew}{Skew}

\theoremstyle{plain}
\newtheorem{theorem}{Theorem}[section]
\newtheorem{proposition}[theorem]{Proposition}
\newtheorem{lemma}[theorem]{Lemma}
\newtheorem{corollary}[theorem]{Corollary}
\newtheorem{mainresult}{Theorem}
\newtheorem{maincorollary}[mainresult]{Corollary}

\theoremstyle{definition}
\newtheorem{definition}[theorem]{Definition}
\theoremstyle{remark}
\newtheorem{remark}[theorem]{Remark}
\newtheorem{example}[theorem]{Example}

\title[Harmonicity and Existence of Algebraic Solitons]{Harmonicity and Existence of Algebraic Generalized Ricci Solitons}
\author{Huibin Chen}
\address{Institute of Mathematics, School of Mathematical Sciences, Nanjing Normal University, Nanjing 210023, P.R. China}
\email{chenhuibin@njnu.edu.cn}
\author{Zhiqi Chen}
\address{School of Mathematics and Statistics, Guangdong University of Technology, Guangzhou 510520, P.R. China}
\email{chenzhiqi@gdut.edu.cn}
\author{Fuhai Zhu}
\address{Department of Mathematics, Nanjing University, Nanjing 210023, P.R. China}
\email{zhufuhai@nju.edu.cn}
\date{}
\hypersetup{pdftitle={Harmonicity and Existence of Algebraic Generalized Ricci Solitons},pdfauthor={Huibin Chen, Zhiqi Chen, Fuhai Zhu},pdfsubject={Harmonicity criteria and existence and nonexistence of generalized nilsolitons},pdfkeywords={generalized Ricci flow, harmonic torsion, Killing form, generalized nilsoliton, Einstein nilradical}}

\begin{document}

\begin{abstract}
In this paper, we study the harmonicity and existence of algebraic generalized Ricci solitons. Firstly, we characterize harmonic torsion of algebraic generalized Ricci solitons on arbitrary metric Lie algebras by an identity involving the Killing form. In particular, positive semidefiniteness of the Killing form implies harmonicity without a unimodularity assumption. Then, we construct generalized nilsolitons with nonzero torsion on indecomposable three-step nilpotent Lie algebras admitting no classical nilsoliton, in dimension seven and in dimensions $6m+d$ for $m>d\geq1$. Furthermore, we provide a spectral obstruction for nilpotent Lie algebras with abelian derived algebra and an obstruction based on the action of derivations on the quotient by the center. Combining these obstructions with explicit constructions, we prove that the filiform Lie algebra $\mathfrak m_2(n)$, $n\geq5$, admits a generalized nilsoliton if and only if $5\leq n\leq8$.
\end{abstract}
\maketitle

\section{Introduction}

The generalized Ricci flow couples the Ricci flow of a Riemannian metric to the evolution of a closed three-form. It belongs to the differential geometry of the generalized tangent bundle, where complex and symplectic structures can be treated within a common framework \cite{Hi2003,Gu2011}. The underlying Courant algebroid originates in the work of Courant and of Liu, Weinstein and Xu \cite{Co1990,LiWeXu1997}; exact Courant algebroids and their cohomological classification are discussed in \cite{Se2017}. The interpretation of curvature and its evolution in this setting also connects the flow with T-duality and renormalization group equations \cite{Ga2019,SeVa2017,SeVa2020}. We use the conventions developed systematically in \cite{GaSt2021}. Thus a metric $g(t)$ and a closed three-form $H(t)$ evolve by
\begin{equation}\label{eq:intro-flow}
\left\{
\begin{aligned}
\frac{\partial g}{\partial t}&=-2\Ric_g+\frac12H^2,\\
\frac{\partial H}{\partial t}&=\Delta_gH,\\
dH&=0.
\end{aligned}
\right.
\end{equation}
where $H^2(X,Y)$ is the full contraction of $\iota_XH$ and $\iota_YH$, and $\Delta_g=-(dd_g^*+d_g^*d)$. Setting $H=0$, it is precisely the Hamilton's Ricci flow \cite{Ha1982}.

The three-form in \eqref{eq:intro-flow} is the torsion of a metric connection. In Hermitian geometry, the Bismut connection is the canonical Hermitian connection with totally skew-symmetric torsion \cite{Bi1989}. If $g$ is Hermitian and $H$ is its Bismut torsion, closedness of $H$ relates the generalized Ricci flow, by a time-dependent diffeomorphism, to the pluriclosed flow introduced by Streets and Tian \cite{StTi2010,StTi2013}. This relation provides a substantial class of geometric examples, while also distinguishing the general system \eqref{eq:intro-flow} from flows constrained by a fixed complex structure. Invariant pluriclosed flows have been studied on nilmanifolds \cite{EnFiVe2015}, on closed complex surfaces \cite{Bo2016}, on two-step nilpotent and almost-abelian Lie groups \cite{ArLa2019}, and on further solvable classes \cite{FiPa2022,FuVe2024}.

Symmetry is equally useful for studying the stationary equations. Compact Hermitian manifolds with flat Bismut connection were classified by Wang, Yang and Zheng \cite{WaYaZh2020}. In the broader Riemannian setting, Podest\`a and Raffero constructed compact homogeneous examples whose Bismut connections are Ricci flat but non-flat \cite{PoRa2023,PoRa2024}, and Lauret and Will obtained further families on compact homogeneous spaces \cite{LaWi2023}. Related variational and stability questions have been investigated for generalized Ricci solitons \cite{Le2024} and for pluriclosed metrics through holomorphic Courant algebroids \cite{GaJoSt2023}.

For left-invariant geometry on a Lie group, the evolution equations reduce to a finite-dimensional system. Lauret's bracket flow realizes the classical homogeneous Ricci flow by varying the Lie bracket on a fixed Euclidean vector space \cite{La2011,La2013}. Paradiso applied a bracket formulation to generalized Ricci flow on nilpotent Lie groups \cite{Pa2021}. Fusi, Lafuente and Stanfield subsequently developed a formulation using Dorfman brackets and the action of both changes of basis and $B$-fields \cite{FuLaSt2024}. Their approach gives long-time existence on solvmanifolds and a geometric notion of generalized Ricci soliton that allows simultaneous scaling of the metric and torsion. Their later work proves subconvergence of rescaled flows to generalized Ricci solitons on simply connected Lie groups with positive-semidefinite Killing form \cite{FuLaSt2026}. These results make the structure and existence of invariant solitons natural questions in their own right.

In the nilpotent case, the classical comparison is particularly precise. A nilsoliton metric satisfies $\Ric=\lambda\Id+D$ for a derivation $D$, and a nilpotent Lie algebra admits such a metric precisely when it is an Einstein nilradical \cite{La2001}. The study of these algebras is closely related to the structure of Einstein solvmanifolds \cite{He1998}, the variational interpretation of the Ricci operator \cite{La2002}, and the pre-Einstein derivation \cite{Ni2011}. Low-dimensional classifications provide explicit examples and obstructions \cite{Wi2003,Fe2012}. More generally, critical points of moment maps and distinguished orbits give an effective framework for the existence problem \cite{Ja2012,BoLa2020}. 

In the present paper, we are devoted to deal with the algebraic generalized Ricci solitons given in \cite{FuLaSt2024}. On a metric Lie algebra $(\glyg,\mu,\bar g)$, their equations in the preferred splitting are
\begin{equation*}
\left\{
\begin{aligned}
\Ric_\mu-\frac14H^2&=\lambda\Id+D,\\
\Delta_\mu H&=-2\lambda H+\rho(D)H,
\end{aligned}
\right.
\end{equation*}
where $D\in\Der(\mu)$ and $d_\mu H=0.$
On a nilpotent Lie group we call such a pair a \emph{generalized nilsoliton}. 

Our first result concerns the harmonicity of the torsion. We call $H$ \emph{harmonic} if $d_\mu H=0$ and $d_g^*H=0$. The question whether generalized nilsolitons have harmonic torsion was posed in \cite[Question~7.12]{FuLaSt2024} and answered by Chen, Chen and Zhu \cite[Theorem~1.1]{ChChZh2026}. Fusi, Lafuente and Stanfield proved the same conclusion for unimodular Lie groups with positive-semidefinite Killing form \cite[Corollary~B]{FuLaSt2026}. We obtain a criterion without a unimodularity assumption.

Denoted by $B_\mu$ for the Killing operator and write $U_\mu$ for the vector defined by $\bar g(U_\mu,X)=\tr(\ad_\mu X)$.  Let $\CEadj$ denote the Chevalley--Eilenberg adjoint with respect to the usual exterior-form inner products. The Riemannian codifferential restricted to invariant forms satisfies $\geomcod=\CEadj+\iota_{U_\mu}$ (see \cite[equation~(4)]{FuLaSt2026}). All norms below use full tensor contractions. For a two-form $\alpha$, set $\alpha(X,Y)=\bar g(K_\alpha X,Y)$,  and define
\begin{equation*}
\mathcal Q_\mu(\alpha):=\frac16|d_\mu\alpha|^2-2\tr(K_\alpha^*K_\alpha M_\mu),
\end{equation*}
where $M_\mu$ is the bracket moment-map operator in \eqref{eq:moment-map-operator}.

\begin{mainresult}\label{thm:criterion-introduction}
Let $(\glyg,\mu,\bar g,H)$ be an algebraic generalized Ricci soliton and set $\eta=\geomcod H$. Then
\begin{equation}
\tr(K_{\eta}^*K_{\eta}B_\mu)=-\frac16|d_\mu\eta|^2-\mathcal Q_\mu(\eta)-\frac12\tr(K_{\eta}^*K_{\eta}H^2)-2|\iota_{U_\mu}\eta|^2\leq0.
\end{equation}
Equality holds if and only if $H$ is harmonic.
\end{mainresult}

Our proof starts from the identity $\rho(R)\eta+\geomcod d_\mu\eta=0$, where $R=\Ric_\mu-\frac14H^2$ and $\eta=\geomcod H$ (see \cite[equation~(53)]{FuLaSt2024}).  The estimate $\mathcal Q_\mu\geq0$ is due to \cite[Lemma~3.1]{ChChZh2026}.

Since $\eta=\geomcod H$, harmonicity follows if $\tr(K_\alpha^*K_\alpha B_\mu)\geq0$ for every $\alpha$ in the subspace
\begin{equation}
\mathcal C^2_{\mu,\bar g}
=\im\bigl(\geomcod:\Lambda^3\glyg^*\longrightarrow\Lambda^2\glyg^*\bigr).
\end{equation}

\begin{maincorollary}\label{coroB}
An algebraic generalized Ricci soliton has harmonic torsion if for any $\alpha\in\mathcal C^2_{\mu,\bar g}$,
\begin{equation}\label{eq:restricted-positivity-introduction}
\tr(K_\alpha^*K_\alpha B_\mu)\geq0.
\end{equation}
In particular, this conclusion holds whenever $B_\mu\geq0$, and hence on every completely solvable Lie group.
\end{maincorollary}

An example (Example~\ref{ex:restricted-positivity-rotating-product}) in Section~\ref{sec:ordinary-harmonicity} shows that \eqref{eq:restricted-positivity-introduction} can hold even when the Killing form has a negative direction. Thus, this condition is strictly weaker than requiring the Killing form to be positive semidefinite.

We next consider whether a given nilpotent Lie algebra admits a generalized nilsoliton. A soliton obtained as a rescaled flow limit may lie on a different Lie group \cite{La2011,FuLaSt2024,FuLaSt2026}. Fusi, Lafuente and Stanfield \cite[Section~7]{FuLaSt2024} discussed the possibility of constructing generalized nilsolitons on nilpotent Lie algebras admitting no classical nilsoliton. We construct such examples in dimension seven and in infinite families of indecomposable nilpotent Lie algebras.

We also give a criterion for the existence of a soliton metric diagonal in a prescribed nice basis $E_1,\ldots,E_n$, with the bracket $\mu$ and the three-form $H_0=\sum_{I\in\mathcal I}h_IE^I$ fixed. Here $\mathcal I$ consists of the triples with $h_I\neq0$. Assume that the subspaces spanned by $\{E_i:i\in I\}$ and $\{E_i:i\notin I\}$ are Lie subalgebras for each $I\in\mathcal I$, and that distinct triples intersect in at most one index. For every diagonal metric, the subalgebra condition makes $H_0$ harmonic, the intersection condition makes $H_0^2$ diagonal, and the nice basis makes the Ricci operator diagonal \cite{LaWi2013}.

Under these assumptions, for $(\mu,H_0)\neq(0,0)$, a diagonal generalized nilsoliton metric with torsion $H_0$ exists if and only if $Gx=\mathbf1$ has a solution with all entries strictly positive. Here $G$ is the Gram matrix of the bracket and three-form weights in \eqref{eq:combined-weight-data}. Adapting the classical convexity arguments \cite{Pa2010,Ni2011,Fe2013}, we realize a soliton metric for the fixed bracket and three-form.

\begin{mainresult}\label{thmC}
For every pair of integers $m>d\geq1$, there exists an indecomposable real three-step nilpotent Lie algebra of dimension $6m+d$ which admits no classical nilsoliton but admits a generalized nilsoliton with nonzero torsion. Consequently, such an algebra exists in every dimension $n\geq43$.
\end{mainresult}

The seven-dimensional example uses the Lie algebra in \cite[Example~4]{Fe2012}. For the general family, unit-norm tight frames determine the Ricci operator. Derivation and centroid calculations exclude all classical nilsoliton metrics and prove indecomposability, respectively.

For a nontrivial generalized nilsoliton, the expanding property and derivation identity in \cite[Proposition~7.7]{FuLaSt2024} imply that the soliton derivation has positive trace. This excludes nonabelian characteristically nilpotent Lie algebras, whose derivations are all nilpotent \cite{Bu2002}. We obtain two further obstructions that hold for every choice of metric and closed three-form.

\begin{mainresult}\label{thmD}
Let $\nalg$ be a nonabelian real nilpotent Lie algebra.
\begin{enumerate}
\item Suppose that $W=[\nalg,\nalg]$ is abelian, and set $r=\dim W$. If $(g,H)$ is a generalized nilsoliton with $c=-\lambda>0$ and soliton derivation $D$, then
\begin{equation}
-3\tr(D^2)+c\bigl(3\tr D+2\tr(D|_W)\bigr)-2rc^2>0.
\end{equation}
Consequently, $\bigl(3\tr D+2\tr(D|_W)\bigr)^2>24r\tr(D^2)$.
\item If every derivation of $\nalg$ that is diagonalizable over $\mathbb R$ induces the zero map on $\nalg/Z(\nalg)$, then $\nalg$ admits no generalized nilsoliton, including the case $H=0$.
\end{enumerate}
\end{mainresult}

The first assertion uses the positive partial Ricci trace over $W$. The second follows from the metric and torsion equations under the decomposition $\nalg=\ker D\oplus\im D$. It also excludes generalized nilsolitons on $\mathfrak c\oplus\mathbb R^k$ for every nonabelian characteristically nilpotent Lie algebra $\mathfrak c$ and every $k\geq0$. For $k\geq1$, the derivation $0\oplus\Id_{\mathbb R^k}$ has positive trace, so these algebras are not excluded by the positive-trace requirement alone.

We apply these methods to determine which members of the positively graded filiform family $\mathfrak m_2(n)$ admit generalized nilsolitons. For $n\geq5$, the nonzero brackets are
\begin{equation}\label{eq:m2-introduction}
[E_1,E_i]=E_{i+1}\quad(2\leq i\leq n-1), \quad [E_2,E_j]=E_{j+2}\quad(3\leq j\leq n-2).
\end{equation}
These algebras admit no classical nilsoliton for $n\geq7$ \cite{Pa2010,Ni2008}. In the generalized nilsoliton setting, we obtain existence in two additional dimensions, namely seven and eight.

\begin{mainresult}\label{thm:filiform-threshold-introduction}
For $n\geq5$, the Lie algebra $\mathfrak m_2(n)$ admits a generalized nilsoliton if and only if $5\leq n\leq8$. In each of these dimensions, there exists a generalized nilsoliton with nonzero torsion.
\end{mainresult}

We give explicit solutions in dimensions five, six and seven. In dimension eight, an explicit polynomial determines the parameters. We prove the existence of the required root and verify positivity and the soliton equations. For $n\geq10$, the inequality for Lie algebras with abelian derived algebra excludes all generalized nilsolitons. In dimension nine, this inequality and the torsion equation leave only one possible weight in the third Chevalley--Eilenberg cohomology, and we show that the corresponding weight space is zero. Thus, nonexistence holds for every metric and closed three-form. 

This paper is organized as follows. In Section~\ref{sec:algebraic-setup}, we fix the conventions. In Section~\ref{sec:ordinary-harmonicity}, we prove the identity involving the Killing form and the harmonicity criterion. Section~\ref{sec:weights} gives criteria for harmonic decomposable forms and diagonal soliton metrics. Section~\ref{sec:tight-frame-existence} presents the seven-dimensional example and the general constructions. In Section~\ref{sec:intrinsic-obstructions}, we establish the nonexistence criteria. Finally, in Section~\ref{sec:filiform-threshold}, we combine these criteria with explicit constructions to determine the existence range for $\mathfrak m_2(n)$.

\subsection*{Declaration of AI use}
During the preparation of this manuscript, the authors used GPT-5.6 Sol and GPT-6 Astra through OpenAI Codex to assist with manuscript organization, language editing, LaTeX preparation, and mathematical verification. These models also contributed to the development of proofs of some propositions. All AI-generated proofs have been checked in detail and verified by the authors. The authors take full responsibility for all mathematical claims and the final content of the manuscript.

\section{Definitions, algebraic framework and conventions}
\label{sec:algebraic-setup}

In this section, we fix the tensor and codifferential conventions, recall the algebraic soliton equations, and record the variation and Ricci identities used in the proofs.

Let $G$ be a simply connected Lie group with Lie algebra $(\glyg,\mu)$ and a left-invariant metric $g$. Write $\bar g=g_e$ and identify left-invariant tensors with their values at the identity.

For $\omega,\tau\in\Lambda^k\glyg^*$, set $\bar g(\omega,\tau)=\sum_{i_1,\ldots,i_k}\omega_{i_1\cdots i_k}\tau_{i_1\cdots i_k}$ and $|\omega|^2=\bar g(\omega,\omega)$, where the components are taken in a $\bar g$-orthonormal basis. This full contraction satisfies $\bar g(\omega,\tau)=k!\langle\omega,\tau\rangle_{\wedge}$, where $\langle\cdot,\cdot\rangle_{\wedge}$ is the usual exterior-form inner product. For the dual basis, write $e^{i_1\cdots i_k}=e^{i_1}\wedge\cdots\wedge e^{i_k}$. For endomorphisms, $A^*=A^t$ denotes the $\bar g$-adjoint.

For a two-form $\alpha$, define $K_\alpha$ by $\alpha(X,Y)=\bar g(K_\alpha X,Y)$. Then $K_\alpha^*=-K_\alpha$ and $K_\alpha^*K_\alpha=-K_\alpha^2$. For a three-form $H$, define $H^2$ by $\bar g(H^2X,Y)=\bar g(\iota_XH,\iota_YH)$. In an orthonormal basis,
\begin{equation}\label{eq:tensor-square-components}
(K_\alpha^*K_\alpha)_{ij}=\sum_p\alpha_{ip}\alpha_{jp},
\qquad
(H^2)_{ij}=\sum_{p,q}H_{ipq}H_{jpq}.
\end{equation}
In particular, $\tr(K_\alpha^*K_\alpha)=|\alpha|^2$ and $\tr(H^2)=|H|^2$. We also write $H^2$ for the corresponding symmetric two-tensor and set $\ker H=\{X\in\glyg:\iota_XH=0\}$.

The actions of $h\in\mathrm{GL}(\glyg)$ on forms and brackets are $h\cdot\omega=(h^{-1})^*\omega$ and $(h\cdot\mu)(X,Y)=h\mu(h^{-1}X,h^{-1}Y)$. Their differentials are
\begin{equation}\label{eq:representations}
\begin{aligned}
(\rho(A)\omega)(X_1,\ldots,X_k)
&=-\sum_{r=1}^k\omega(X_1,\ldots,AX_r,\ldots,X_k),\\
(\theta(A)\mu)(X,Y)
&=A\mu(X,Y)-\mu(AX,Y)-\mu(X,AY).
\end{aligned}
\end{equation}
In particular, $\rho(\Id)\omega=-k\omega$ for $\omega\in\Lambda^k\glyg^*$, and $\theta(\Id)\mu=-\mu$.

\begin{lemma}\label{lem:representation-contractions}
Let $A\in\Sym(\glyg)$. For $\alpha\in\Lambda^2\glyg^*$, we have $\bar g(\rho(A)\alpha,\alpha)=-2\tr(K_\alpha^*K_\alpha A)$. For $H\in\Lambda^3\glyg^*$, we have $\bar g(\rho(A)H,H)=-3\tr(H^2A)$.
\end{lemma}

\begin{proof}
Choose an orthonormal eigenbasis $Ae_i=a_i e_i$. The components of $\rho(A)\alpha$ and $\rho(A)H$ are $-(a_i+a_j)\alpha_{ij}$ and $-(a_i+a_j+a_k)H_{ijk}$. Full contraction gives the two identities.
\end{proof}

On invariant forms, the exterior differential is the Chevalley--Eilenberg differential $d_\mu$, determined by $(d_\mu\xi)(X,Y)=-\xi(\mu(X,Y))$ for $\xi\in\glyg^*$ and the graded Leibniz rule. The Jacobi identity gives $d_\mu^2=0$. Let $\CEadj$ be the adjoint of $d_\mu$ with respect to $\langle\cdot,\cdot\rangle_{\wedge}$. For $\alpha\in\Lambda^p\glyg^*$ and $\beta\in\Lambda^{p+1}\glyg^*$, the full-contraction normalization gives
\begin{equation}\label{eq:CE-adjoint-pairing}
\bar g(\CEadj\beta,\alpha)
=\frac1{p+1}\bar g(\beta,d_\mu\alpha).
\end{equation}

Define $U_\mu$ by $\bar g(U_\mu,X)=\tr(\ad_\mu X)$, where $\ad_\mu X(Y)=\mu(X,Y)$. Write $\geomcod$ for the Riemannian codifferential restricted to left-invariant forms. Then \cite[equation~(4)]{FuLaSt2026} gives
\begin{equation}\label{eq:geometric-codifferential-pairing}
\bar g(\geomcod\beta,\alpha)
=\frac1{p+1}\bar g(\beta,d_\mu\alpha)
+\bar g(\iota_{U_\mu}\beta,\alpha).
\end{equation}
Hence $\geomcod=\CEadj+\iota_{U_\mu}$. The two codifferentials agree on all invariant forms if and only if $U_\mu=0$, equivalently if the Lie algebra is unimodular. Both satisfy $(\CEadj)^2=(\geomcod)^2=0$.

We use the geometric Laplacian $\Delta_\mu=-(d_\mu\geomcod+\geomcod d_\mu)$. An invariant form $\omega$ is harmonic if $d_\mu\omega=0$ and $\geomcod\omega=0$. On a non-unimodular Lie algebra, $\Delta_\mu\omega=0$ alone need not imply harmonicity.

To compute the variation of the geometric codifferential, keep $\bar g$ fixed and let $\nu$ be any skew-symmetric bilinear map. Define $d_\nu$ by the same formula on one-forms and the graded Leibniz rule, and set $\bar g(U_\nu,X)=\tr(Y\mapsto\nu(X,Y))$. With $d_\nu^\dagger$ denoting the exterior-form adjoint, put $d_{g,\nu}^*=d_\nu^\dagger+\iota_{U_\nu}$. These operators depend linearly on $\nu$, so they are defined even when $\nu$ does not satisfy the Jacobi identity.

For $A\in\mathfrak{gl}(\glyg)$, differentiating the equivariance of $d_\mu$ gives $d_{\theta(A)\mu}=[\rho(A),d_\mu]$, where $[P,Q]=P\circ Q-Q\circ P$. Since $\rho(A)^*=\rho(A^*)$ on each exterior power, taking adjoints gives $d_{\theta(A)\mu}^\dagger=[\CEadj,\rho(A^*)]$. Moreover, $\ad_{\theta(A)\mu}X=[A,\ad_\mu X]-\ad_\mu(AX)$ implies $U_{\theta(A)\mu}=-A^*U_\mu$. The slot action of $\rho$ also gives $[\iota_{U_\mu},\rho(A^*)]=-\iota_{A^*U_\mu}$. Combining these identities, we obtain
\begin{equation}\label{eq:geometric-codifferential-variation}
d_{g,\theta(A)\mu}^{*}
=[\CEadj,\rho(A^*)]-\iota_{A^*U_\mu}
=[\geomcod,\rho(A^*)].
\end{equation}
This is the identity in \cite[equation~(36)]{FuLaSt2024}, with the contribution of $U_\mu$ included explicitly.

The soliton equations use the geometric Laplacian $\Delta_\mu$. We follow the convention of \cite[Section~3]{FuLaSt2024}, which permits simultaneous scaling of the metric and torsion. For a closed three-form $H$, the symmetric Ricci tensor of the metric connection with torsion $H$ is $\Ric_g-\frac14H^2$ (see \cite[Section~2.3]{FuLaSt2024}). Raising an index, set $\RicB_{\mu,H}=\Ric_\mu-\frac14H^2$.

\begin{definition}[{\cite[Proposition~7.4]{FuLaSt2024}}]
Let $H\in\Lambda^3\glyg^*$ satisfy $d_\mu H=0$. With $\bar g$ fixed, $(\mu,H)$ is an algebraic generalized Ricci soliton if there exist $\lambda\in\R$ and $D\in\Der(\mu)$ such that
\begin{equation}\label{eq:soliton-system}
\left\{
\begin{aligned}
\RicB_{\mu,H}&=\lambda\Id+D,\\
\Delta_\mu H&=-2\lambda H+\rho(D)H.
\end{aligned}
\right.
\end{equation}
\end{definition}

The first equation makes $D$ self-adjoint. Since $\rho(\Id)H=-3H$, the second is equivalent to $\Delta_\mu H=\lambda H+\rho(\RicB_{\mu,H})H$.

Since $G$ is simply connected, $D$ generates automorphisms $\varphi_s$ with $(d\varphi_s)_e=e^{sD}$. Their generator $X_D$ satisfies $\mathcal L_{X_D}g=2g(D\cdot,\cdot)$ and $\mathcal L_{X_D}H=-\rho(D)H$. Thus \eqref{eq:soliton-system} defines a geometric generalized Ricci soliton (see \cite[Proposition~3.4 and Lemma~7.5]{FuLaSt2024}).

The soliton is expanding, steady or shrinking according as $\lambda<0$, $\lambda=0$ or $\lambda>0$. We include the trivial pair $(\mu,H)=(0,0)$ with $\lambda=0$ and $D=0$.

To express the Ricci operator in terms of the bracket, define $M_\mu\in\Sym(\glyg)$ for any skew-symmetric bilinear map $\mu$ by
\begin{equation}\label{eq:moment-map-operator}
\bar g(M_\mu X,Y)
=-\frac12\sum_i\bar g\bigl(\mu(X,e_i),\mu(Y,e_i)\bigr)
+\frac14\sum_{i,j}\bar g\bigl(\mu(e_i,e_j),X\bigr)\bar g\bigl(\mu(e_i,e_j),Y\bigr),
\end{equation}
where $\{e_i\}$ is orthonormal. On bracket tensors, set $\bar g(\mu,\nu)=\sum_{i,j}\bar g(\mu(e_i,e_j),\nu(e_i,e_j))$ and $|\mu|^2=\bar g(\mu,\mu)$. Pairing the three terms of $\theta(A)\mu$ with $\mu$ gives
\begin{equation}\label{eq:moment-map-identity}
4\tr(M_\mu A)=\bar g(\theta(A)\mu,\mu),
\quad A\in\Sym(\glyg),
\end{equation}
with the normalization in \cite[Section~4.3]{FuLaSt2024}. Taking $A=\Id$ gives $\tr(M_\mu)=-|\mu|^2/4$. These identities do not require the Jacobi identity.

For a Lie bracket $\mu$, define the Killing operator by $\bar g(B_\mu X,Y)=\tr(\ad_\mu X\circ\ad_\mu Y)$ and set $S_\mu=\Sym(\ad_\mu U_\mu)$, where $\Sym(A)=(A+A^*)/2$. By \cite[Section~2.3]{La2013},
\begin{equation}\label{eq:general-Ricci-decomposition}
\Ric_\mu=M_\mu-\frac12B_\mu-S_\mu.
\end{equation}
For unimodular Lie algebras, $U_\mu=S_\mu=0$. For nilpotent Lie algebras, $B_\mu=0$ as well, so $\Ric_\mu=M_\mu$.

\section{The Killing-form identity and harmonicity criteria}
\label{sec:ordinary-harmonicity}

In this section, we express $\mathcal Q_\mu$ as a sum of squares and derive an identity involving the Killing form. We then prove the harmonicity criterion and its consequences for nilpotent and completely solvable Lie groups.

Let $(V,\bar g)$ be a Euclidean vector space and let $\mu\in\Lambda^2V^*\otimes V$. For $\alpha\in\Lambda^2V^*$, set 
\begin{equation*}
\mathcal Q_\mu(\alpha)=\frac16|d_\mu\alpha|^2-2\tr(K_\alpha^*K_\alpha M_\mu).
\end{equation*}
The estimate $\mathcal Q_\mu\geq0$ and its coordinate sum-of-squares proof are due to \cite[Lemma~3.1]{ChChZh2026}. The proof does not require the Jacobi identity, as noted in \cite[Remark~3.2]{ChChZh2026}. We give an invariant formula and describe its equality condition.

Write $T_\mu(X,Y,Z)=\bar g(\mu(X,Y),Z)$ and $K=K_\alpha$. Let $\mathcal K_r$ insert $K$ into the $r$-th slot of a covariant three-tensor, and set $(\mathsf S_{13}F)(X,Y,Z)=F(Z,Y,X)$. All tensor norms are full contractions. For endomorphisms, write $\Skew(A)=(A-A^*)/2$ and $\|A\|_{\mathrm{HS}}^2=\tr(A^*A)$.

\begin{lemma}
For every $\mu\in\Lambda^2V^*\otimes V$ and $\alpha\in\Lambda^2V^*$,
\begin{equation}\label{eq:moment-map-estimate}
\begin{aligned}
\mathcal Q_\mu(\alpha)
=\frac12\left|\mathcal K_1T_\mu-\mathsf S_{13}\mathcal K_1T_\mu\right|^2=2\sum_i\left\|\Skew\bigl(\ad_\mu(e_i)\circ K_\alpha\bigr)\right\|_{\mathrm{HS}}^2\geq0,
\end{aligned}
\end{equation}
where $\{e_i\}$ is any orthonormal basis and $\ad_\mu(Y)X=\mu(Y,X)$. Equality holds if and only if $\ad_\mu(Y)\circ K_\alpha$ is self-adjoint for every $Y\in V$.
\end{lemma}

\begin{proof}
Set $T=T_\mu$, $\mathcal U=\mathcal K_1T$ and $\mathcal W=\mathcal K_3T$. Since $K^*=-K$, we have $\mathcal K_r^*=-\mathcal K_r$. Apply \eqref{eq:moment-map-identity} with $A=K_\alpha^*K_\alpha=-K^2$. Lowering the output index in \eqref{eq:representations} gives $T_{\theta(A)\mu}=\mathcal K_3^*\mathcal K_3T-\mathcal K_1^*\mathcal K_1T-\mathcal K_2^*\mathcal K_2T$. As $T$ is skew-symmetric in its first two slots, $|\mathcal K_2T|=|\mathcal U|$. Therefore
\begin{equation}\label{eq:moment-map-trace-identity}
-2\tr(K_\alpha^*K_\alpha M_\mu)=|\mathcal U|^2-\frac12|\mathcal W|^2.
\end{equation}

Define $(\mathsf PF)(X,Y,Z)=F(Y,Z,X)$. Then $\mathsf P$ is orthogonal, $\mathsf P^3=\Id$ and $\mathsf P^*=\mathsf P^2$. The definition of $d_\mu$ and the identity $\alpha(v,Z)=-\bar g(v,KZ)$ give $d_\mu\alpha=\mathcal W+\mathsf P\mathcal W+\mathsf P^2\mathcal W$. Expanding the squared norm, we obtain
\begin{equation}\label{eq:two-form-differential-identity}
\frac16|d_\mu\alpha|^2=\frac12|\mathcal W|^2+\bar g(\mathcal W,\mathsf P\mathcal W).
\end{equation}

The skew-symmetry of $T$ gives $\mathcal K_3\mathsf P\mathcal K_3T=-\mathcal K_1\mathsf S_{13}\mathcal K_1T$. Evaluating at $(X,Y,Z)$, this is $T(Y,KZ,KX)=-T(KZ,Y,KX)$. Taking adjoints, we have
\begin{equation}\label{eq:mixed-term}
\bar g(\mathcal W,\mathsf P\mathcal W)=-\bar g(T,\mathcal K_3\mathsf P\mathcal K_3T)=\bar g(T,\mathcal K_1\mathsf S_{13}\mathcal K_1T)=-\bar g(\mathcal U,\mathsf S_{13}\mathcal U).
\end{equation}
Combining \eqref{eq:moment-map-trace-identity}, \eqref{eq:two-form-differential-identity} and \eqref{eq:mixed-term} gives $\mathcal Q_\mu(\alpha)=|\mathcal U|^2-\bar g(\mathcal U,\mathsf S_{13}\mathcal U)$. Since $\mathsf S_{13}$ is an orthogonal involution, we have $\mathcal Q_\mu(\alpha)=\frac12|\mathcal U-\mathsf S_{13}\mathcal U|^2$.

For $Y\in V$, set $A_Y=\ad_\mu(Y)\circ K_\alpha$. Then $\mathcal U(X,Y,Z)=-\bar g(A_YX,Z)$ and $(\mathsf S_{13}\mathcal U)(X,Y,Z)=-\bar g(A_Y^*X,Z)$. Contracting in the first and third slots gives $|\mathcal U-\mathsf S_{13}\mathcal U|^2=4\sum_i\|\Skew(A_{e_i})\|_{\mathrm{HS}}^2$. This proves \eqref{eq:moment-map-estimate} and the equality condition.
\end{proof}

We now combine the square formula with the geometric codifferential and the Ricci decomposition.

\begin{proposition}
Let $(\glyg,\mu,\bar g,H)$ be an algebraic generalized Ricci soliton and set $\eta=\geomcod H$. Then
\begin{equation}\label{eq:exact-energy-decomposition}
0=\frac1{12}|d_\mu\eta|^2+\frac12\mathcal Q_\mu(\eta)+\frac12\tr(K_\eta^*K_\eta B_\mu)+\frac14\tr(K_\eta^*K_\eta H^2)+|\iota_{U_\mu}\eta|^2.
\end{equation}
Moreover, for every orthonormal basis $\{e_j\}$,
\begin{equation}\label{eq:H-square-nonnegative}
\tr(K_\eta^*K_\eta H^2)=\sum_j|\iota_{K_\eta^*e_j}H|^2\geq0.
\end{equation}
\end{proposition}

\begin{proof}
Set $R=\Ric_\mu-\frac14H^2=\lambda\Id+D$ and recall that $\eta=\geomcod H$. Since $D$ is a derivation, $\theta(D)\mu=0$. Together with $\theta(\Id)\mu=-\mu$, this gives $\theta(R)\mu=-\lambda\mu$. By \eqref{eq:geometric-codifferential-variation} and $R^*=R$, we have $[\geomcod,\rho(R)]=d_{g,\theta(R)\mu}^*=-\lambda\geomcod$. The last equality uses the linear dependence of $d_{g,\nu}^*$ on $\nu$.

Since $\rho(\Id)H=-3H$, we have $\rho(R)H=-3\lambda H+\rho(D)H$. Thus the second soliton equation becomes $\Delta_\mu H=\lambda H+\rho(R)H$. On the other hand, $d_\mu H=0$ gives $\Delta_\mu H=-d_\mu\eta$. Therefore $\rho(R)H+d_\mu\eta=-\lambda H$.

Apply $\geomcod$ to this equation. The commutator identity gives $\geomcod\rho(R)H=\rho(R)\eta-\lambda\eta$. Cancelling $-\lambda\eta$ from both sides, we obtain
\begin{equation}\label{eq:geometric-defect-operator}
\rho(R)\eta+\geomcod d_\mu\eta=0.
\end{equation}
This is the operator identity in \cite[equation~(53)]{FuLaSt2024}, with the geometric codifferential used throughout.

We next compute the contribution of $U_\mu$ to the pairing with $\eta$. Since $(\geomcod)^2=0$, we have $\geomcod\eta=0$. Apply \eqref{eq:geometric-codifferential-pairing} with $\beta=\eta$ and $\alpha=\iota_{U_\mu}\eta$, of degrees two and one. We obtain $0=\frac12\bar g(\eta,d_\mu\iota_{U_\mu}\eta)+|\iota_{U_\mu}\eta|^2$.

On invariant forms, our convention for $\rho$ gives $\mathcal L_{U_\mu}\eta=\rho(\ad_\mu U_\mu)\eta$. Cartan's formula therefore reads $\iota_{U_\mu}d_\mu\eta=\rho(\ad_\mu U_\mu)\eta-d_\mu\iota_{U_\mu}\eta$. The skew-adjoint part of $\ad_\mu U_\mu$ acts skew-adjointly on two-forms, so its pairing with $\eta$ itself vanishes. Hence Lemma~\ref{lem:representation-contractions}, applied to $S_\mu=\Sym(\ad_\mu U_\mu)$, gives $\bar g(\rho(\ad_\mu U_\mu)\eta,\eta)=-2\tr(K_\eta^*K_\eta S_\mu)$. Taking the inner product of Cartan's formula with $\eta$ now gives
\begin{equation}\label{eq:modular-pairing}
\bar g(\eta,\iota_{U_\mu}d_\mu\eta)
=-2\tr(K_\eta^*K_\eta S_\mu)+2|\iota_{U_\mu}\eta|^2.
\end{equation}

We now pair \eqref{eq:geometric-defect-operator} with $\eta$. Since $R$ is self-adjoint, Lemma~\ref{lem:representation-contractions} gives $\bar g(\rho(R)\eta,\eta)=-2\tr(K_\eta^*K_\eta R)$. Apply \eqref{eq:geometric-codifferential-pairing} again, now with $\beta=d_\mu\eta$ and $\alpha=\eta$, of degrees three and two. It gives $\bar g(\geomcod d_\mu\eta,\eta)=\frac13|d_\mu\eta|^2+\bar g(\eta,\iota_{U_\mu}d_\mu\eta)$. Substituting \eqref{eq:modular-pairing}, we obtain
\begin{equation}\label{eq:pre-energy}
0=-2\tr\bigl(K_\eta^*K_\eta(R+S_\mu)\bigr)
+\frac13|d_\mu\eta|^2+2|\iota_{U_\mu}\eta|^2.
\end{equation}

The Ricci decomposition \eqref{eq:general-Ricci-decomposition} gives $R=M_\mu-\frac12B_\mu-S_\mu-\frac14H^2$. Thus the term $-S_\mu$ in $R$ cancels the additional $S_\mu$ in \eqref{eq:pre-energy}. Substitute this decomposition and divide by two. The terms involving $d_\mu\eta$ and $M_\mu$ satisfy $\frac16|d_\mu\eta|^2-\tr(K_\eta^*K_\eta M_\mu)=\frac1{12}|d_\mu\eta|^2+\frac12\mathcal Q_\mu(\eta)$ by the definition of $\mathcal Q_\mu$. This proves \eqref{eq:exact-energy-decomposition}.

Finally, cyclicity gives $\tr(K_\eta^*K_\eta H^2)=\tr(K_\eta H^2K_\eta^*)$. In an orthonormal basis $\{e_j\}$, the latter trace is $\sum_j\bar g(H^2K_\eta^*e_j,K_\eta^*e_j)$. By the definition of $H^2$, each summand equals $|\iota_{K_\eta^*e_j}H|^2$. This proves \eqref{eq:H-square-nonnegative}.
\end{proof}

In the identity \eqref{eq:exact-energy-decomposition}, only the term $\frac12\tr(K_\eta^*K_\eta B_\mu)$ is possibly nonnegative. The equality case gives the harmonicity criterion.

\begin{theorem}[Theorem~\ref{thm:criterion-introduction}]\label{thm:killing-harmonicity}
Let $(\glyg,\mu,\bar g,H)$ be an algebraic generalized Ricci soliton and set $\eta=\geomcod H$. Then
\begin{equation}\label{eq:killing-harmonicity-criterion}
\tr(K_\eta^*K_\eta B_\mu)=-\frac16|d_\mu\eta|^2-\mathcal Q_\mu(\eta)-\frac12\tr(K_\eta^*K_\eta H^2)-2|\iota_{U_\mu}\eta|^2\leq0.
\end{equation}
Equality holds if and only if $H$ is harmonic.
\end{theorem}

\begin{proof}
Solving \eqref{eq:exact-energy-decomposition} for $\tr(K_\eta^*K_\eta B_\mu)$ gives \eqref{eq:killing-harmonicity-criterion}. By \eqref{eq:moment-map-estimate} and \eqref{eq:H-square-nonnegative}, every term on the right-hand side is nonpositive. If $H$ is harmonic, then $\eta=0$ and equality holds.

Conversely, suppose $\tr(K_\eta^*K_\eta B_\mu)\geq0$. Every term in \eqref{eq:exact-energy-decomposition} is then nonnegative, so each term vanishes. In particular, $d_\mu\eta=0$ and $\tr(K_\eta^*K_\eta H^2)=0$. By \eqref{eq:H-square-nonnegative}, the latter equality gives $\iota_{K_\eta^*e_j}H=0$ for every vector in an orthonormal basis $\{e_j\}$. Since $K_\eta^*=-K_\eta$, it follows that $\im K_\eta\subset\ker H$, where $\ker H=\{X\in\glyg:\iota_XH=0\}$.

It remains to show that $\eta=0$. Apply \eqref{eq:geometric-codifferential-pairing} with $\beta=H$ and $\alpha=\eta$. Since $\geomcod H=\eta$, we obtain $|\eta|^2=\frac13\bar g(H,d_\mu\eta)+\bar g(\eta,\iota_{U_\mu}H)$. The first term vanishes because $d_\mu\eta=0$. For the second term, write $\eta_{ij}=\eta(e_i,e_j)$. Full contraction gives $\bar g(\eta,\iota_{U_\mu}H)=\sum_{i,j}\eta_{ij}H(U_\mu,e_i,e_j)$. Since $K_\eta e_i=\sum_j\eta_{ij}e_j$, this equals $\sum_iH(U_\mu,e_i,K_\eta e_i)$. Each summand vanishes because $K_\eta e_i\in\ker H$ and $H$ is alternating. Therefore $|\eta|^2=0$. Together with $d_\mu H=0$, this proves that $H$ is harmonic.
\end{proof}

\begin{remark}
Fusi, Lafuente and Stanfield proved that algebraic generalized Ricci solitons on simply connected unimodular Lie groups with positive-semidefinite Killing form have harmonic torsion \cite[Corollary~B]{FuLaSt2026}. Since $B_\mu\geq0$ implies $\tr(K_\eta^*K_\eta B_\mu)\geq0$, Theorem~\ref{thm:killing-harmonicity} gives the same conclusion without the unimodularity assumption.
\end{remark}

\begin{corollary}[Corollary~\ref{coroB}]
Let $(\glyg,\mu,\bar g)$ be a metric Lie algebra and set $\mathcal C^2_{\mu,\bar g}=\im(\geomcod:\Lambda^3\glyg^*\longrightarrow\Lambda^2\glyg^*)$. Suppose that for every $\alpha\in\mathcal C^2_{\mu,\bar g}$
\begin{equation}\label{eq:restricted-killing-positivity}
\tr(K_\alpha^*K_\alpha B_\mu)\geq0.
\end{equation}
Then every algebraic generalized Ricci soliton $(\glyg,\mu,\bar g,H)$ has harmonic torsion. In particular, this conclusion holds when $B_\mu\geq0$, and hence on every completely solvable Lie group.
\end{corollary}

\begin{proof}
For an algebraic generalized Ricci soliton, $\eta=\geomcod H$ belongs to $\mathcal C^2_{\mu,\bar g}$ by definition. Thus \eqref{eq:restricted-killing-positivity} gives $\tr(K_\eta^*K_\eta B_\mu)\geq0$. By \eqref{eq:killing-harmonicity-criterion}, this trace is also nonpositive, so it is zero. The equality assertion in Theorem~\ref{thm:killing-harmonicity} shows that $H$ is harmonic.

If $B_\mu\geq0$, its positive-semidefinite square root is defined. By cyclicity of the trace, $\tr(K_\alpha^*K_\alpha B_\mu)=\|K_\alpha B_\mu^{1/2}\|_{\mathrm{HS}}^2\geq0$ for every two-form $\alpha$. Hence \eqref{eq:restricted-killing-positivity} holds.

Finally, suppose $\glyg$ is completely solvable. For every $X\in\glyg$, the eigenvalues $r_1,\ldots,r_n$ of $\ad_\mu X$ are real. Counting them with algebraic multiplicity, we have $\bar g(B_\mu X,X)=\tr((\ad_\mu X)^2)=\sum_jr_j^2\geq0$. Thus $B_\mu\geq0$, and the preceding argument applies.
\end{proof}

\begin{corollary}[{\cite[Theorem~1.1]{ChChZh2026}}]\label{cor:nilpotent-harmonicity}
Every algebraic generalized Ricci soliton on a nilpotent Lie group has harmonic torsion.
\end{corollary}

\begin{proof}
Nilpotency gives $B_\mu=0$, so Theorem~\ref{thm:criterion-introduction} applies.
\end{proof}

The positivity condition on $\mathcal C^2_{\mu,\bar g}$ is strictly weaker than $B_\mu\geq0$, as the following example shows.

\begin{example}\label{ex:restricted-positivity-rotating-product}
Let $\{e_1,\ldots,e_6\}$ be orthonormal, with nonzero brackets $[e_1,e_2]=e_2+2e_3$ and $[e_1,e_3]=-2e_2+e_3$. Set $H=\frac2{\sqrt3}e^{456}$, $\lambda=-2$ and $D=\diag(0,0,0,\frac43,\frac43,\frac43)$. These data define an expanding algebraic generalized Ricci soliton. Its Killing form has a negative direction, whereas $\tr(K_\alpha^*K_\alpha B_\mu)=0$ for every $\alpha\in\mathcal C^2_{\mu,\bar g}$.
\end{example}

\begin{proof}
Write $\glyg=\mathfrak s\oplus V$ as an orthogonal direct sum of ideals, where $\mathfrak s=\operatorname{span}\{e_1,e_2,e_3\}$ and $V=\operatorname{span}\{e_4,e_5,e_6\}$. The definition of $M_\mu$ gives $M_\mu|_{\mathfrak s}=\diag(-5,0,0)$. We also have $U_\mu=2e_1$, $B_\mu|_{\mathfrak s}=\diag(-6,0,0)$ and $S_\mu|_{\mathfrak s}=\diag(0,2,2)$. The endomorphisms $M_\mu,B_\mu,S_\mu$ vanish on $V$. The Ricci decomposition gives $\Ric_\mu=-2\Id_{\mathfrak s}\oplus0_V$, while $H^2=0_{\mathfrak s}\oplus\frac83\Id_V$. Hence $\Ric_\mu-\frac14H^2=-2\Id+D$. The product form $H$ is closed and coclosed, $D$ is a symmetric derivation, and $\rho(D)H=-4H=2\lambda H$. Thus the soliton equations hold.

To compute $\mathcal C^2_{\mu,\bar g}$, set $\delta=\geomcod=d_\mu^\dagger+\iota_{2e_1}$. From $d_\mu e^2=-e^{12}+2e^{13}$ and $d_\mu e^3=-2e^{12}-e^{13}$, we obtain $\delta e^1=2$ and $\delta e^2=\delta e^3=0$. On two-forms, $\delta e^{12}=e^2-2e^3$, $\delta e^{13}=2e^2+e^3$ and $\delta e^{23}=0$. Moreover, $d_\mu^\dagger e^{123}=-2e^{23}$ and $\iota_{2e_1}e^{123}=2e^{23}$, so $\delta e^{123}=0$.

Since $\delta$ acts only on the $\mathfrak s$ factor, these formulas give $\mathcal C^2_{\mu,\bar g}=(\operatorname{span}\{e^2,e^3\}\wedge V^*)\oplus\Lambda^2V^*$. Every two-form in this space annihilates $e_1$. Therefore $\tr(K_\alpha^*K_\alpha B_\mu)=-6|K_\alpha e_1|^2=0$ on $\mathcal C^2_{\mu,\bar g}$, although $\bar g(B_\mu e_1,e_1)=-6$.
\end{proof}

\section{Harmonic forms and diagonal soliton metrics}\label{sec:weights}

In this section, we record the harmonic reduction and trace identities for generalized nilsolitons. We then characterize harmonic decomposable forms and give a criterion for diagonal soliton metrics with the Lie bracket and three-form fixed.

Throughout this section, the Lie algebra is nilpotent, so $\Ric_\mu=M_\mu$ and $\geomcod=\CEadj$. An algebraic generalized Ricci soliton on a nilpotent Lie group is called a \emph{generalized nilsoliton}. We call $(\mu,H)$ \emph{nontrivial} if $(\mu,H)\neq(0,0)$. A nilpotent Lie algebra admits a generalized nilsoliton if some inner product and closed three-form satisfy \eqref{eq:soliton-system}.

\begin{corollary}\label{cor:reduced-soliton-system}
Let $(\nalg,\mu,\bar g)$ be a metric nilpotent Lie algebra and let $H\in\Lambda^3\nalg^*$ be closed. Then $(\mu,H)$ is a generalized nilsoliton if and only if $\CEadj H=0$ and there exist $\lambda\in\R$ and $D\in\Der(\mu)\cap\Sym(\nalg)$ such that
\begin{equation}\label{eq:reduced-soliton-system}
\left\{
\begin{aligned}
\Ric_\mu-\frac14H^2&=\lambda\Id+D,\\
\rho(D)H&=2\lambda H.
\end{aligned}
\right.
\end{equation}
\end{corollary}

\begin{proof}
By \cite[Theorem~1.1]{ChChZh2026}, recalled in Corollary~\ref{cor:nilpotent-harmonicity}, every generalized nilsoliton has harmonic torsion. Hence $\Delta_\mu H=0$ and \eqref{eq:soliton-system} reduces to \eqref{eq:reduced-soliton-system}. Conversely, closedness and coclosedness give $\Delta_\mu H=0$, so the reduced equations imply the soliton system.
\end{proof}

In an orthonormal eigenbasis $De_i=d_i e_i$, the torsion equation in \eqref{eq:reduced-soliton-system} says that $H_{ijk}\neq0$ implies $d_i+d_j+d_k=-2\lambda$.

\begin{lemma}\label{lem:ricci-derivation-orthogonality}
For a metric nilpotent Lie algebra, $\tr(\Ric_\mu E)=0$ for every $E\in\Der(\mu)$.
\end{lemma}

\begin{proof}
The moment-map identity \eqref{eq:moment-map-identity} holds for arbitrary endomorphisms. Since $\Ric_\mu=M_\mu$ and $\theta(E)\mu=0$ for a derivation, it follows that $4\tr(\Ric_\mu E)=\bar g(\theta(E)\mu,\mu)=0$.
\end{proof}

The expanding property is known \cite[Proposition~7.7]{FuLaSt2024}. We reprove it together with the derivation trace identity.

\begin{proposition}\label{prop:nilsoliton-trace}
Let $(\nalg,\mu,\bar g,H)$ be a nontrivial generalized nilsoliton with constant $\lambda$ and symmetric derivation $D$. Then $\lambda<0$, $D\neq0$, and
\begin{equation}\label{eq:nilsoliton-trace}
\|D\|_{\mathrm{HS}}^2=-\lambda\left(\tr D-\frac16|H|^2\right),
\qquad
\tr D>\frac16|H|^2.
\end{equation}
\end{proposition}

\begin{proof}
Set $R=\Ric_\mu-\frac14H^2=\lambda\Id+D$. Since $D$ is a derivation, $\theta(D)\mu=0$. Together with $\theta(\Id)\mu=-\mu$, this gives $\theta(R)\mu=-\lambda\mu$. The torsion equation in \eqref{eq:reduced-soliton-system} gives $\rho(D)H=2\lambda H$. Since $\rho(\Id)H=-3H$, we also obtain $\rho(R)H=-\lambda H$.

Applying \eqref{eq:moment-map-identity} with the self-adjoint endomorphism $R$, we obtain $4\tr(M_\mu R)=\bar g(\theta(R)\mu,\mu)=-\lambda|\mu|^2$. Similarly, Lemma~\ref{lem:representation-contractions} gives $-3\tr(H^2R)=\bar g(\rho(R)H,H)=-\lambda|H|^2$. Since $\Ric_\mu=M_\mu$, we have $\|R\|_{\mathrm{HS}}^2=\tr(R^2)=\tr(M_\mu R)-\frac14\tr(H^2R)$. Substituting the preceding two identities yields $\|R\|_{\mathrm{HS}}^2=-\lambda(|\mu|^2/4+|H|^2/12)$.

To determine the sign of $\lambda$, use $\tr(M_\mu)=-|\mu|^2/4$ and $\tr(H^2)=|H|^2$. These give $\tr R=-(|\mu|^2+|H|^2)/4<0$ because $(\mu,H)\neq(0,0)$. Hence $R\neq0$. The factor $|\mu|^2/4+|H|^2/12$ is also positive, so the identity for $\|R\|_{\mathrm{HS}}^2$ implies $\lambda<0$.

Next, take the trace pairing of the metric equation with $D$. By Lemma~\ref{lem:ricci-derivation-orthogonality}, $\tr(\Ric_\mu D)=0$. Since $D$ is self-adjoint, $\tr(D^2)=\|D\|_{\mathrm{HS}}^2$, and therefore $\lambda\tr D+\|D\|_{\mathrm{HS}}^2=-\frac14\tr(H^2D)$. Pairing $\rho(D)H=2\lambda H$ with $H$ and applying Lemma~\ref{lem:representation-contractions} gives $-3\tr(H^2D)=2\lambda|H|^2$. Thus $-\frac14\tr(H^2D)=\lambda|H|^2/6$. Substitution proves the equality in \eqref{eq:nilsoliton-trace}.

It remains to prove that $D\neq0$. Suppose $D=0$. The torsion equation becomes $2\lambda H=0$, so $H=0$ because $\lambda<0$. Nontriviality then gives $\mu\neq0$, and the metric equation reduces to $\Ric_\mu=\lambda\Id$. Since a nonzero nilpotent Lie algebra has nontrivial center, choose a nonzero central vector $Z$ and an orthonormal basis $\{e_i\}$. The terms involving $\mu(Z,e_i)$ in \eqref{eq:moment-map-operator} vanish, so $\bar g(\Ric_\mu Z,Z)=\frac14\sum_{i,j}\bar g(\mu(e_i,e_j),Z)^2\geq0$. This contradicts $\bar g(\Ric_\mu Z,Z)=\lambda|Z|^2<0$. Hence $D\neq0$.

Finally, the equality in \eqref{eq:nilsoliton-trace} gives $\tr D-\frac16|H|^2=\|D\|_{\mathrm{HS}}^2/(-\lambda)>0$, which proves the strict inequality.
\end{proof}

In particular, a nilpotent Lie algebra whose derivations are all traceless admits no nontrivial generalized nilsoliton.

We next characterize harmonic decomposable forms. Neither subspace in the following decomposition is required to be an ideal.

\begin{lemma}\label{lem:harmonic-subalgebra-support}
Let $\nalg=V\oplus W$ be an orthogonal decomposition of a metric nilpotent Lie algebra, and let $\omega$ be a nonzero volume form of $V$, extended by zero on arguments in $W$. Then $\omega$ is harmonic if and only if $V$ and $W$ are Lie subalgebras.
\end{lemma}

\begin{proof}
Set $p=\dim V$ and $n=\dim\nalg$. Fix an orientation of $\nalg$ and let $*$ be the associated metric Hodge star. By the Hodge formula for the geometric codifferential, $\omega$ is harmonic if and only if $d_\mu\omega=0$ and $d_\mu(*\omega)=0$. If $p=0$ or $p=n$, the forms $\omega$ and $*\omega$ are constants or volume forms, so both are closed. The subalgebra conditions also hold automatically. We may therefore assume $0<p<n$.

Suppose first that $\omega$ is harmonic. Choose $w_1,w_2\in W$ and $v_1,\ldots,v_{p-1}\in V$. In the Chevalley--Eilenberg formula for $(d_\mu\omega)(w_1,w_2,v_1,\ldots,v_{p-1})$, every term except the one involving $[w_1,w_2]$ retains at least one argument in $W$ and hence vanishes. Thus $0=-\omega(\operatorname{pr}_V[w_1,w_2],v_1,\ldots,v_{p-1})$. Since $\omega|_V$ is a nonzero volume form and the vectors $v_1,\ldots,v_{p-1}$ are arbitrary, we obtain $\operatorname{pr}_V[w_1,w_2]=0$. Hence $[W,W]\subset W$.

By orthogonality, $*\omega$ is a nonzero multiple of a volume form of $W$, extended by zero on arguments in $V$. Applying the preceding argument to $*\omega$ and using $d_\mu(*\omega)=0$, we obtain $[V,V]\subset V$.

Conversely, suppose that $V$ and $W$ are subalgebras. We first prove $d_\mu\omega=0$. Components with all arguments in $V$ vanish because $\dim V=p$. For a component with at least two arguments in $W$, every term in the differential also vanishes. Indeed, the bracket of two $W$ vectors lies in $W$, while every other term retains an argument in $W$.

It remains to consider one argument in $W$. For $w\in W$, set $A_w=\operatorname{pr}_V\circ\ad_w|_V$. Since $\omega$ vanishes on arguments in $W$, the differential formula gives $(d_\mu\omega)(w,v_1,\ldots,v_p)=-\sum_{r=1}^p\omega(v_1,\ldots,A_wv_r,\ldots,v_p)$ for $v_1,\ldots,v_p\in V$. Evaluating on a basis of $V$, only the diagonal coefficients of $A_w$ contribute. Their sum is $\tr(A_w)$, so the expression equals $-\tr(A_w)\omega(v_1,\ldots,v_p)$.

Because $w\in W$ and $W$ is a subalgebra, $W$ is invariant under $\ad_w$. The diagonal blocks of $\ad_w$ relative to $V\oplus W$ are therefore $A_w$ and $\ad_w|_W$. Consequently, $\tr(A_w)=\tr(\ad_w)-\tr(\ad_w|_W)$. Nilpotency of $\nalg$ implies that $\ad_w$ and its restriction to $W$ are nilpotent, so both traces vanish. Thus $\tr(A_w)=0$, and $d_\mu\omega=0$ follows.

Interchanging $V$ and $W$ in the same argument proves $d_\mu(*\omega)=0$. Hence $\omega$ is also coclosed and is therefore harmonic.
\end{proof}
Fix a basis $E_1,\ldots,E_n$ with dual basis $E^1,\ldots,E^n$. For an increasing triple $I=\{i_1,i_2,i_3\}$, set $E^I=E^{i_1}\wedge E^{i_2}\wedge E^{i_3}$ and $V_I=\operatorname{span}\{E_i:i\in I\}$. Write $I^c=\{1,\ldots,n\}\setminus I$. Every metric diagonal in this basis makes $V_I$ and $V_{I^c}$ orthogonal, and $E^I$ is a nonzero multiple of the metric volume form of $V_I$, extended by zero on $V_{I^c}$. Lemma~\ref{lem:harmonic-subalgebra-support} therefore shows that $E^I$ is harmonic if and only if both subspaces are subalgebras. This condition is independent of the diagonal metric.

If each nonzero summand of a linear combination $\sum_I h_IE^I$ satisfies these subalgebra conditions, then the sum is harmonic for every diagonal metric. For a fixed metric, however, a sum may be harmonic even when its individual summands are not, because their differentials or codifferentials can cancel.

For the Gram criterion below, we require the Ricci operator and $H^2$ to be diagonal. We first recall the nice basis condition, which ensures diagonality of the Ricci operator. Following \cite[Definition~3]{Ni2011}, a basis $E_1,\ldots,E_n$ is called \emph{nice} if every nonzero bracket $[E_i,E_j]$ is a multiple of one basis vector and distinct nonzero brackets with the same output have disjoint input pairs. Here bracket pairs are counted only with $i<j$. Choose a metric $\bar g$ diagonal in this nice basis and set $e_i=E_i/\sqrt{\bar g(E_i,E_i)}$. Rescaling individual basis vectors preserves niceness, so $e_1,\ldots,e_n$ is an orthonormal nice basis.

Since $\Ric_\mu=M_\mu$, we may use \eqref{eq:moment-map-operator} to compute its off-diagonal entries. For $r\neq s$, the second sum vanishes because a bracket has at most one nonzero output component. A nonzero term $\bar g([e_r,e_i],[e_s,e_i])$ in the first sum would give two nonzero brackets with the same output and a common input index $i$, contrary to niceness. Hence the Ricci operator is diagonal.

Enumerate the nonzero bracket components by $a=1,\ldots,m$. For the $a$-th component $[e_i,e_j]=c_{ij}^k e_k$, with $i<j$, put $c_a=c_{ij}^k$ and $\alpha_a=E_{kk}-E_{ii}-E_{jj}$, where $E_{rr}$ is the diagonal projection onto $\R e_r$. For a diagonal endomorphism $A=\diag(t_1,\ldots,t_n)$, the $a$-th component of $\theta(A)\mu$ is $(t_k-t_i-t_j)c_a=\tr(A\alpha_a)c_a$.

Nilpotency gives $k\notin\{i,j\}$. Indeed, $k=j$ would give a nonzero eigenvalue of $\ad_{e_i}$, while $k=i$ would give one of $\ad_{e_j}$. In \eqref{eq:moment-map-operator}, the first sum contributes $-c_a^2/2$ in each input direction. The two ordered pairs $(i,j)$ and $(j,i)$ in the second sum together contribute $c_a^2/2$ in the output direction. Thus
\begin{equation}\label{eq:nice-Ricci}
\Ric_\mu=\frac12\sum_{a=1}^m c_a^2\alpha_a.
\end{equation}

Let $e^1,\ldots,e^n$ be the dual coframe and write a three-form as $H=\sum_{r=1}^s h_r e^{I_r}$, where the triples $I_r$ are distinct and increasing and every $h_r\neq0$. The sum is empty when $H=0$. Set $\beta_r=-\sum_{i\in I_r}E_{ii}$. For a diagonal endomorphism $A$, the form action satisfies $\rho(A)e^{I_r}=\tr(A\beta_r)e^{I_r}$. By \eqref{eq:tensor-square-components}, $(H^2)_{ii}=\sum_{p,q}H_{ipq}^2=2\sum_{r:\,i\in I_r}h_r^2$. The factor two comes from the two possible orders of the other two indices in each triple.

For $i\neq j$, a nonzero product $H_{ipq}H_{jpq}$ in $(H^2)_{ij}$ requires two distinct triples sharing the indices $p,q$. Hence $|I_r\cap I_t|\leq1$ for $r\neq t$ ensures that $H^2$ is diagonal. Whenever $H^2$ is diagonal, its diagonal entries determine the full operator and give $H^2=-2\sum_{r=1}^s h_r^2\beta_r$.

We now collect the bracket and three-form weights and their actual coefficient squares. Define
\begin{equation}\label{eq:combined-weight-data}
\begin{aligned}
(\Gamma_1,\ldots,\Gamma_N)&=(\alpha_1,\ldots,\alpha_m,\beta_1,\ldots,\beta_s),
\qquad N=m+s,\\
z&=(c_1^2,\ldots,c_m^2,h_1^2,\ldots,h_s^2)^{\mathsf T},\\
G_{\mu,H}&=(\tr(\Gamma_p\Gamma_q))_{p,q=1}^N.
\end{aligned}
\end{equation}
The weights are self-adjoint diagonal endomorphisms, so $\tr(\Gamma_p\Gamma_q)$ is their Hilbert--Schmidt inner product. Thus $G_{\mu,H}$ is their Gram matrix. It depends only on the indices of the nonzero bracket and three-form components, while $z$ records their actual squared coefficients in the chosen orthonormal basis.

If $H^2$ is diagonal, its expression above and \eqref{eq:nice-Ricci} give
\begin{equation}\label{eq:combined-moment-map}
\Ric_\mu-\frac14H^2=\frac12\sum_{p=1}^N z_p\Gamma_p.
\end{equation}
The next criterion tests the vector $z$ determined by the given bracket, metric and three-form in this basis, as in the classical criterion of \cite[Theorem~1]{Pa2010}. Write $\one_N$ for the all-ones vector in $\R^N$.

\begin{proposition}\label{thm:augmented}
Let $(\nalg,\mu,\bar g)$ be a metric nilpotent Lie algebra with an orthonormal nice basis. Suppose that $H$ is closed, $H^2$ is diagonal in this basis, and $(\mu,H)\neq(0,0)$. Let $z$ and $G_{\mu,H}$ be as in \eqref{eq:combined-weight-data}. Then $(\mu,H)$ is a generalized nilsoliton if and only if, for some $\kappa>0$,
\begin{equation}\label{eq:augmented-gram}
\CEadj H=0,
\qquad
G_{\mu,H}z=\kappa\one_N.
\end{equation}
In that case, $\lambda=-\kappa/2$ and $D=\Ric_\mu-\frac14H^2+\frac\kappa2\Id$.
\end{proposition}

\begin{proof}
Set $R=\Ric_\mu-\frac14H^2$. By \eqref{eq:combined-moment-map}, $R$ is diagonal and $(G_{\mu,H}z)_p=2\tr(R\Gamma_p)$. For any $\lambda\in\R$, the diagonal map $D=R-\lambda\Id$ is a derivation if and only if $\tr(D\alpha_a)=0$ for every bracket component. The torsion equation $\rho(D)H=2\lambda H$ holds if and only if $\tr(D\beta_r)=2\lambda$ for every form component. Since $\tr\alpha_a=-1$ and $\tr\beta_r=-3$, these two conditions are equivalent to $\tr(R\Gamma_p)=-\lambda$ for every $p$, or $G_{\mu,H}z=-2\lambda\one_N$.

For a generalized nilsoliton, Corollary~\ref{cor:reduced-soliton-system} gives $\CEadj H=0$, and Proposition~\ref{prop:nilsoliton-trace} gives $\lambda<0$. Thus \eqref{eq:augmented-gram} holds with $\kappa=-2\lambda$. Conversely, if \eqref{eq:augmented-gram} holds, set $\lambda=-\kappa/2$ and $D=R-\lambda\Id$. The preceding equivalences give both equations in \eqref{eq:reduced-soliton-system}, and Corollary~\ref{cor:reduced-soliton-system} applies.
\end{proof}

Coclosedness cannot be omitted from Proposition~\ref{thm:augmented}. On $\mathfrak h_3\oplus\R$, take an orthonormal coframe with $d_\mu e^3=-e^{12}$ and all other $d_\mu e^i=0$, and set $H=e^{124}$ \cite[Example~5.5]{ChChZh2026}. This form is closed and $H^2$ is diagonal. The Gram matrix is $G=\left(\begin{smallmatrix}3&2\\2&3\end{smallmatrix}\right)$ and $z=(1,1)^{\mathsf T}$, so $Gz=5\one_2$. However, $d_\mu e^{34}=-e^{124}$ gives $\CEadj H=-e^{34}\neq0$.

We now fix the Lie bracket and three-form and allow the diagonal metric to vary. Under the following hypotheses, an abstract positive Gram solution determines a metric for which the actual coefficient squares satisfy \eqref{eq:augmented-gram}. The proof uses the convexity arguments underlying the classical criteria \cite{Pa2010,Ni2011,Fe2013}. We identify diagonal endomorphisms with Euclidean $\R^n$, with inner product $\langle A,B\rangle=\tr(AB)$.

\begin{theorem}
Fix a nilpotent Lie bracket $\mu$ in a nice basis $E_1,\ldots,E_n$, with dual basis $E^1,\ldots,E^n$. Let $\mathcal I$ be a set of triples such that $V_I$ and $V_{I^c}$ are subalgebras for every $I\in\mathcal I$, and $|I\cap J|\leq1$ for distinct $I,J\in\mathcal I$. Fix $H_0=\sum_{I\in\mathcal I}h_IE^I$ with all listed coefficients nonzero, and assume $(\mu,H_0)\neq(0,0)$. The set $\mathcal I$ may be empty. Let $\Gamma_1,\ldots,\Gamma_N$ and $G$ be the weights and Gram matrix defined in \eqref{eq:combined-weight-data}. The following conditions are equivalent:
\begin{enumerate}
\item There exists a positive-definite metric $g$, diagonal in the prescribed basis, such that $(g,H_0)$ is a generalized nilsoliton on $(\nalg,\mu)$.
\item The system $Gx=\one_N$ has a solution with all entries strictly positive.
\item The point of $C=\operatorname{conv}\{\Gamma_1,\ldots,\Gamma_N\}$ closest to the origin belongs to $\operatorname{relint}C$.
\end{enumerate}
\end{theorem}

\begin{proof}
For every metric diagonal in the prescribed basis, Lemma~\ref{lem:harmonic-subalgebra-support} makes $H_0$ harmonic. Niceness makes the Ricci operator diagonal, and the intersection condition makes $H_0^2$ diagonal. If the metric is a soliton, Proposition~\ref{thm:augmented} gives $Gz=\kappa\one_N$ for its actual coefficient squares $z>0$ and some $\kappa>0$. Taking $x=z/\kappa$ proves (1)$\Rightarrow$(2).

Set $\mathcal U=\operatorname{span}\{\Gamma_p-\Gamma_q\}$. Suppose $Gx=\one_N$ with $x>0$. Set $s=\sum_p x_p$ and $b=s^{-1}\sum_p x_p\Gamma_p$. Then $b\in\operatorname{relint}C$ and $\langle\Gamma_p,b\rangle=1/s$ for every $p$. Taking the convex combination with coefficients $x_p/s$ gives $|b|^2=1/s$, so $b\perp\mathcal U$. Consequently, $|y|^2=|b|^2+|y-b|^2$ for every $y\in C$. This proves (2)$\Rightarrow$(3).

It remains to prove (3)$\Rightarrow$(1). Let $b$ be the closest point, and suppose $b\in\operatorname{relint}C$. Varying $b$ in both directions along $\mathcal U$ gives $b\perp\mathcal U$, and $\operatorname{aff}C=b+\mathcal U$. Also, $b\neq0$, since every weight has coordinate sum $-1$ or $-3$ and hence $0\notin C$.

Let $g_0$ make $E_1,\ldots,E_n$ orthonormal, and let $z_p^0>0$ be the corresponding coefficient squares. For $a=(a_1,\ldots,a_n)\in\R^n$, define $g_a(E_i,E_j)=e^{2a_i}\delta_{ij}$. Its orthonormal basis is $e_i=e^{-a_i}E_i$. Thus the bracket coefficient $c_{ij}^k$ is multiplied by $e^{a_k-a_i-a_j}$, while the form coefficient $h_I$ is multiplied by $e^{-\sum_{i\in I}a_i}$. The actual coefficient squares of the fixed pair $(\mu,H_0)$ with metric $g_a$ are therefore
\begin{equation}\label{eq:diagonal-coefficient-scaling}
z_p(a)=z_p^0\exp\bigl(2\langle\Gamma_p,a\rangle\bigr),
\qquad 1\leq p\leq N.
\end{equation}
We seek $a$ for which these coefficients satisfy the Gram equation. On $\mathcal U$, define
\begin{equation}\label{eq:diagonal-convex-function}
f(a)=\sum_{p=1}^N z_p^0\exp\bigl(2\langle\Gamma_p-b,a\rangle\bigr).
\end{equation}
Since $b\in\operatorname{relint}C$, the convex hull of the vectors $\Gamma_p-b$ contains zero in its interior relative to $\mathcal U$. If $\mathcal U\neq0$, compactness of its unit sphere gives $\epsilon>0$ such that $\max_p\langle\Gamma_p-b,u\rangle\geq\epsilon$ for every unit $u\in\mathcal U$. Hence $f(a)\geq(\min_p z_p^0)e^{2\epsilon|a|}$, so $f$ attains a minimum on $\mathcal U$. If $\mathcal U=0$, take $a=0$.

Choose a minimizer $a\in\mathcal U$. Since $b\perp\mathcal U$, the summands in \eqref{eq:diagonal-convex-function} are exactly the actual coefficient squares $z_p(a)$ in \eqref{eq:diagonal-coefficient-scaling}. The gradient lies in $\mathcal U$, so minimality gives $\sum_p z_p(a)(\Gamma_p-b)=0$. Pairing this equation with each $\Gamma_q$ and using $\langle\Gamma_q,b\rangle=|b|^2$, we obtain
\begin{equation}\label{eq:diagonal-realized-gram}
Gz(a)=\kappa\one_N,
\qquad
\kappa=\left(\sum_p z_p(a)\right)|b|^2>0.
\end{equation}
As noted at the start of the proof, $H_0$ is harmonic and both the Ricci operator and $H_0^2$ are diagonal for $g_a$. Thus \eqref{eq:diagonal-realized-gram} and Proposition~\ref{thm:augmented} show that $(g_a,H_0)$ is a generalized nilsoliton on the fixed Lie algebra.
\end{proof}

\section{Generalized nilsolitons on Lie algebras admitting no classical nilsoliton}\label{sec:tight-frame-existence}

In this section, we construct generalized nilsolitons with nonzero torsion on three-step nilpotent Lie algebras admitting no classical nilsoliton. We first define the Lie algebras and exclude classical nilsoliton metrics. We then construct the generalized nilsolitons and prove indecomposability, obtaining a seven-dimensional example and families of dimension $6m+d$ for $m>d\geq1$.

Let $m\geq d\geq1$ be integers, and let $W$ be a real vector space of dimension $d$. Choose nonzero vectors $v_1,\ldots,v_m$ that span $W$. Define
\begin{equation}\label{eq:configuration-bracket}
\begin{gathered}
\nalg(v)=W\oplus\bigoplus_{i=1}^m\operatorname{span}\{X_i,Y_i,Z_i,P_i,Q_i,T_i\}, \\ [X_i,Y_i]=P_i, [Y_i,Z_i]=Q_i, [X_i,Z_i]=v_i, [X_i,Q_i]=T_i, [Z_i,P_i]=-T_i,
\end{gathered}
\end{equation}
with all unlisted brackets zero. Distinct blocks commute, and $W$ and the $T_i$ are central. The only potentially nonzero Jacobi identity is that for $X_i,Y_i,Z_i$, whose two nonzero terms are $T_i$ and $-T_i$. Thus \eqref{eq:configuration-bracket} defines a Lie algebra.

Set $V_1=\bigoplus_i\operatorname{span}\{X_i,Y_i,Z_i\}$, $V_2=W\oplus\bigoplus_i\operatorname{span}\{P_i,Q_i\}$, and $V_3=\bigoplus_i\R T_i$. These subspaces give a positive grading with $[V_r,V_s]\subset V_{r+s}$. Writing $\nalg(v)^1=\nalg(v)$ and $\nalg(v)^{r+1}=[\nalg(v),\nalg(v)^r]$, we have $\nalg(v)^2=V_2\oplus V_3$, $\nalg(v)^3=V_3$, and $\nalg(v)^4=0$. Moreover, $Z(\nalg(v))=W\oplus V_3$. Indeed, commuting a central vector with $Y_i$ eliminates its $X_i$- and $Z_i$-components. Commuting it with $X_i$ and $Z_i$ then eliminates its $Y_i$-, $Q_i$-, and $P_i$-components. In particular, $\nalg(v)$ is three-step nilpotent and has dimension $6m+d$.

\begin{proposition}\label{prop:configuration-not-Einstein}
For every nonzero spanning configuration $v_1,\ldots,v_m$, the Lie algebra $\nalg(v)$ admits no classical nilsoliton metric.
\end{proposition}

\begin{proof}
Set $\nalg=\nalg(v)$. We first determine a pre-Einstein derivation of $\nalg$. Recall that a derivation $\phi$ is pre-Einstein if it is diagonalizable over $\R$ and satisfies $\tr(\phi A)=\tr A$ for every derivation $A$. Define $\phi|_{V_r}=(r/2)\Id_{V_r}$ for $r=1,2,3$. For $x\in V_r$ and $y\in V_s$, the grading gives $\phi[x,y]=((r+s)/2)[x,y]=[\phi x,y]+[x,\phi y]$. Hence $\phi$ is a derivation and is diagonalizable over $\R$.

Choose a basis of $W$ and append the displayed block vectors. Let $A$ be an arbitrary derivation, and denote by $a_F$ the coefficient of $F$ in $AF$ for each basis vector $F$. We do not assume that $A$ preserves the grading subspaces. Applying $A$ to $[Y_i,Z_i]=Q_i$ gives $AQ_i=[AY_i,Z_i]+[Y_i,AZ_i]$. By \eqref{eq:configuration-bracket}, only the $Y_i$-component of $AY_i$ and the $Z_i$-component of $AZ_i$ contribute to the $Q_i$-component. Thus $a_{Q_i}=a_{Y_i}+a_{Z_i}$.

Similarly, applying $A$ to $[X_i,Q_i]=T_i$ gives $AT_i=[AX_i,Q_i]+[X_i,AQ_i]$. The $T_i$-component of the first term is $a_{X_i}T_i$, and that of the second is $a_{Q_i}T_i$. Although $[Z_i,P_i]=-T_i$, this bracket cannot contribute to either term, since one input is fixed as $Q_i$ or $X_i$. Therefore
\begin{equation}\label{eq:configuration-derivation-traces}
a_{Q_i}=a_{Y_i}+a_{Z_i},\qquad
a_{T_i}=a_{X_i}+a_{Q_i}=a_{X_i}+a_{Y_i}+a_{Z_i}.
\end{equation}
The operator $\phi-\Id$ is diagonal in the chosen basis, with eigenvalues $-1/2$, $0$, and $1/2$ on $V_1,V_2,V_3$, respectively. Hence only the diagonal coefficients of $A$ enter $\tr((\phi-\Id)A)$. By \eqref{eq:configuration-derivation-traces}, we obtain $\tr((\phi-\Id)A)=-\frac12\sum_i(a_{X_i}+a_{Y_i}+a_{Z_i})+\frac12\sum_i a_{T_i}=0$. Thus $\tr(\phi A)=\tr A$ for every derivation $A$, and $\phi$ is pre-Einstein.

Suppose now that $\nalg$ admits a classical nilsoliton metric $g$. Since $\nalg$ is nonabelian, Proposition~\ref{prop:nilsoliton-trace} with $H=0$ gives $\Ric=-c\Id+D$ for some $c>0$ and derivation $D$. For every derivation $A$, Lemma~\ref{lem:ricci-derivation-orthogonality} gives $0=\tr(\Ric A)=-c\tr A+\tr(DA)$. It follows that $\tr((D/c)A)=\tr A$. Moreover, $D=\Ric+c\Id$ is symmetric with respect to $g$, so $D/c$ is diagonalizable over $\R$. Therefore $D/c$ is also pre-Einstein. By \cite[Theorem~1(1)(b)]{Ni2011}, there exists $F\in\operatorname{Aut}(\nalg)$ such that $D/c=F\phi F^{-1}$.

Choose $0\neq w_0\in W$ and set $w=Fw_0$. Since $\phi w_0=w_0$, we have $Dw=cF\phi w_0=cw$, and hence $\Ric w=0$. On the other hand, $W\subset Z(\nalg)\cap[\nalg,\nalg]$, and automorphisms preserve both the center and the derived algebra. Thus $0\neq w\in Z(\nalg)\cap[\nalg,\nalg]$.

Let $\{e_a\}$ be a $g$-orthonormal basis. Since $w$ is central, the negative term in \eqref{eq:moment-map-operator} vanishes, giving $g(\Ric w,w)=\frac14\sum_{a,b}g([e_a,e_b],w)^2$. If this sum were zero, then $w$ would be orthogonal to every $[e_a,e_b]$, and hence to $[\nalg,\nalg]$. Since $w\in[\nalg,\nalg]$, this would imply $g(w,w)=0$, contrary to $w\neq0$. Therefore $g(\Ric w,w)>0$, contradicting $\Ric w=0$.
\end{proof}

We now construct generalized nilsoliton metrics on these Lie algebras. Give $W$ a Euclidean inner product $\langle\cdot,\cdot\rangle_0$. A unit-norm tight frame is a collection satisfying
\begin{equation}\label{eq:tight-frame}
|v_i|_0=1,\qquad \sum_{i=1}^m v_i\otimes v_i=\frac md\Id_W,
\end{equation}
where $(v\otimes v)w=\langle v,w\rangle_0v$. For another metric $g$ on $W$, write $(v\otimes_g v)w=g(v,w)v$. The vectors $v_i$ need not lie in coordinate lines, so we compute the complete Ricci operator directly.

\begin{theorem}\label{thm:tight-frame-construction}
Suppose \eqref{eq:tight-frame} holds. Set $a=5m+4d$, $b=4m+3d$, and $c_0=2d$. On \eqref{eq:configuration-bracket}, make the displayed block vectors mutually orthogonal and orthogonal to $W$. Give $X_i,Y_i,Z_i$ length one, give $P_i,Q_i$ squared length $a$, give $T_i$ squared length $ab$, and set $g|_W=c_0\langle\cdot,\cdot\rangle_0$. Let $y_i^*,p_i^*,q_i^*$ be dual to the corresponding normalized block vectors. Then
\begin{equation}\label{eq:frame-torsion}
H=2\sqrt{m+d}\sum_{i=1}^m y_i^*\wedge p_i^*\wedge q_i^*
\end{equation}
is harmonic and defines a generalized nilsoliton with $\lambda=-(10m+9d)$. The Ricci operator, $H^2$, and the symmetric derivation $D$ act by the scalars in Table~\ref{tab:frame-tensors}.
\end{theorem}

\begin{table}[htbp]
\centering
\renewcommand{\arraystretch}{1.25}
\begin{tabular}{cccc}
\hline
Direction & $\Ric$ & $H^2$ & $D$\\
\hline
$x_i,z_i$ & $-9(m+d)/2$ & $0$ & $(11m+9d)/2$\\
$y_i$ & $-(5m+4d)$ & $8(m+d)$ & $3(m+d)$\\
$p_i,q_i$ & $(m+d)/2$ & $8(m+d)$ & $(17m+15d)/2$\\
$t_i$ & $4m+3d$ & $0$ & $14m+12d$\\
$W$ & $m$ & $0$ & $11m+9d$\\
\hline
\end{tabular}
\caption{The operators for the tight-frame metric.}\label{tab:frame-tensors}
\end{table}

\begin{proof}
Set $x_i=X_i$, $y_i=Y_i$, $z_i=Z_i$, $p_i=P_i/\sqrt a$, $q_i=Q_i/\sqrt a$, and $t_i=T_i/\sqrt{ab}$. These vectors are orthonormal and orthogonal to $W$. Complete them to a $g$-orthonormal basis of $\nalg(v)$ by choosing a $g$-orthonormal basis of $W$. Put $\widehat v_i=v_i/\sqrt{c_0}$. Since $g|_W=c_0\langle\cdot,\cdot\rangle_0$ and $|v_i|_0=1$, we have $|\widehat v_i|_g=1$. Rescaling \eqref{eq:configuration-bracket} gives
\begin{equation}\label{eq:frame-orthonormal-brackets}
\begin{aligned}
[x_i,y_i]&=\sqrt a\,p_i,& [y_i,z_i]&=\sqrt a\,q_i,&[x_i,z_i]&=\sqrt{c_0}\,\widehat v_i,\\
[x_i,q_i]&=\sqrt b\,t_i,&[z_i,p_i]&=-\sqrt b\,t_i.
\end{aligned}
\end{equation}
Up to skew-symmetry, these are all the nonzero brackets among the chosen basis vectors.

We first verify harmonicity. For each $i$, the subspace $\operatorname{span}\{y_i,p_i,q_i\}$ is abelian. Its orthogonal complement is also a subalgebra. Indeed, a bracket can have a $p_i$- or $q_i$-component only if one input has a $y_i$-component, and no bracket has a $y_i$-component. The form $y_i^*\wedge p_i^*\wedge q_i^*$ is the volume form of this subspace, extended by zero on its orthogonal complement. Lemma~\ref{lem:harmonic-subalgebra-support} therefore shows that it is harmonic. By linearity, $H$ is harmonic.

Set $h=2\sqrt{m+d}$. The contractions of $H$ with $y_i,p_i,q_i$ are $h\,p_i^*\wedge q_i^*$, $-h\,y_i^*\wedge q_i^*$, and $h\,y_i^*\wedge p_i^*$, respectively. These two-forms are mutually orthogonal, and contractions from different blocks are also orthogonal. All remaining contractions vanish. Thus $H^2$ is diagonal. Under the full-contraction convention, each nonzero contraction has squared norm $2h^2=8(m+d)$, giving the $H^2$ column of Table~\ref{tab:frame-tensors}.

We next compute the complete Ricci operator using \eqref{eq:moment-map-operator}. In its negative term, fix the second input to be one of the chosen basis vectors. The nonzero outputs corresponding to distinct first inputs are orthogonal, as follows directly from \eqref{eq:frame-orthonormal-brackets} and the vanishing of brackets between distinct blocks. Hence the negative term is diagonal. In the positive term, each nonzero bracket has its output in one of $\R p_i$, $\R q_i$, $\R t_i$, or $W$. Therefore mixed entries vanish except possibly within $W$.

For $x_i$ and $z_i$, the sum of squared bracket lengths is $a+b+c_0$, and neither vector occurs as a bracket output. Their Ricci eigenvalue is therefore $-(a+b+c_0)/2$. The corresponding sum for $y_i$ is $2a$, giving eigenvalue $-a$. Each of $p_i,q_i$ contributes $-b/2$ as an input and $a/2$ as an output, giving eigenvalue $(a-b)/2$. Finally, $t_i$ is central and occurs as the output of two unordered input pairs. Each pair contributes $b/2$, so its Ricci eigenvalue is $b$.

Since $W$ is central, only the positive term contributes to $\Ric|_W$. The bracket $[x_i,z_i]=\sqrt{c_0}\,\widehat v_i$ contributes $(c_0/2)\widehat v_i\otimes_g\widehat v_i$. For $w\in W$, the metric rescaling gives $g(\widehat v_i,w)=\sqrt{c_0}\langle v_i,w\rangle_0$, and hence $g(\widehat v_i,w)\widehat v_i=\langle v_i,w\rangle_0v_i$. Thus $\widehat v_i\otimes_g\widehat v_i=v_i\otimes v_i$. Applying \eqref{eq:tight-frame} and $c_0=2d$, we obtain the complete operator
\begin{equation}
\begin{aligned}
\Ric x_i&=-\tfrac12(a+b+c_0)x_i,&\Ric z_i&=-\tfrac12(a+b+c_0)z_i,\\
\Ric y_i&=-ay_i,&\Ric p_i&=\tfrac12(a-b)p_i,\\
\Ric q_i&=\tfrac12(a-b)q_i,&\Ric t_i&=bt_i,\\
\Ric|_W&=\frac{c_0}{2}\sum_i\widehat v_i\otimes_g\widehat v_i
=\frac{c_0m}{2d}\Id_W=m\Id_W.
\end{aligned}
\end{equation}
Substituting $a=5m+4d$, $b=4m+3d$, and $c_0=2d$ gives the Ricci column of Table~\ref{tab:frame-tensors}. For $\lambda=-(10m+9d)$, the symmetric operator $D=\Ric-H^2/4-\lambda\Id$ has the entries in the last column. Thus the metric equation holds.

It remains to verify that $D$ is a derivation and satisfies the torsion equation. Set $u=(11m+9d)/2$ and $s=3(m+d)$. The eigenvalues of $D$ on $x_i,y_i,z_i,p_i,q_i,t_i$ are $u,s,u,u+s,u+s,2u+s$, respectively, and $D|_W=2u\Id_W$. The brackets with outputs $p_i,q_i$ have input eigenvalue sum $u+s$. The bracket with output $\widehat v_i$ has input eigenvalue sum $2u$, and those with output $t_i$ have input eigenvalue sum $u+(u+s)=2u+s$. Thus the derivation identity holds for every nonzero bracket. Since every vector in the chosen basis is an eigenvector of $D$, the identity also holds for every zero bracket. Hence $D$ is a derivation.

The eigenvalue sum on each torsion triple is $s+2(u+s)=20m+18d=-2\lambda$. By the definition of $\rho$, its action on each summand of $H$ is multiplication by the negative of this sum. Therefore $\rho(D)H=2\lambda H$. Together with harmonicity and the metric equation, this proves \eqref{eq:reduced-soliton-system}.
\end{proof}

The case $m=d=1$ gives the following seven-dimensional example. Let $\nalg$ have basis $E_1,\ldots,E_7$ and nonzero brackets
\begin{equation}\label{eq:standard-bracket}
\begin{aligned}
[E_1,E_2]&=E_4,&[E_1,E_3]&=E_5,&[E_2,E_3]&=E_6,\\
[E_1,E_6]&=E_7,&[E_3,E_4]&=-E_7.
\end{aligned}
\end{equation}
This is the Lie algebra $\mathfrak g_{3.1}(i_0)$ in \cite[Example~4, pp.~654--655]{Fe2012}, where it is shown to admit no classical nilsoliton. Proposition~\ref{prop:configuration-not-Einstein} gives an independent proof of this fact.

\begin{proposition}\label{prop:single-torsion-seed}
On \eqref{eq:standard-bracket}, the metric $g=\diag(1,1,1,9,2,9,63)$ in the fixed basis and the three-form $H=18\sqrt2\,E^{246}$ define a generalized nilsoliton with $\lambda=-19$ and $D=\diag(10,6,10,16,20,16,26)$.
\end{proposition}

\begin{proof}
Apply Theorem~\ref{thm:tight-frame-construction} with $m=d=1$, $W=\R E_5$, $v_1=E_5$, and $|E_5|_0=1$. Identify $(X_1,Y_1,Z_1,P_1,Q_1,T_1)$ with $(E_1,E_2,E_3,E_4,E_6,E_7)$. Then $a=9$, $b=7$, and $c_0=2$, giving the stated metric and derivation. In the orthonormal coframe, \eqref{eq:frame-torsion} is $H=2\sqrt2\,e^{246}$. Since $e^2=E^2$, $e^4=3E^4$, and $e^6=3E^6$, this equals $18\sqrt2\,E^{246}$.
\end{proof}

We next determine when the underlying Lie algebra is indecomposable. Call the spanning configuration $v_1,\ldots,v_m$ \emph{linearly indecomposable} if no direct sum $W=W_1\oplus W_2$ with both summands nonzero has every $v_i$ in one of the summands.

\begin{proposition}\label{prop:configuration-indecomposable}
The Lie algebra $\nalg(v)$ is indecomposable if and only if the configuration is linearly indecomposable. For a tight frame, this is equivalent to connectedness of the graph with vertices $1,\ldots,m$ and edges $ij$ whenever $\langle v_i,v_j\rangle_0\neq0$.
\end{proposition}

\begin{proof}
Set $\nalg=\nalg(v)$. Suppose first that the configuration splits over $W=W_1\oplus W_2$, with both summands nonzero. For $r=1,2$, let $\nalg_r$ be the sum of $W_r$ and the six-dimensional blocks whose vectors $v_i$ belong to $W_r$. Each block is assigned to exactly one summand because $v_i\neq0$. By \eqref{eq:configuration-bracket}, $[\nalg_r,\nalg_r]\subset\nalg_r$ and $[\nalg_1,\nalg_2]=0$. Thus $\nalg=\nalg_1\oplus\nalg_2$ is a direct sum of nonzero ideals.

For the converse, suppose that the configuration is linearly indecomposable. Recall that the centroid of $\nalg$ consists of endomorphisms $C$ satisfying $C[x,y]=[Cx,y]=[x,Cy]$ for all $x,y\in\nalg$. A projection onto an ideal direct summand belongs to the centroid and satisfies $C^2=C$. Conversely, for an idempotent $C$ in the centroid, $\nalg=\im C\oplus\ker C$. The centroid identities show that both subspaces are ideals. Moreover, if $u=Cx$ and $v\in\ker C$, then $[u,v]=[Cx,v]=[x,Cv]=0$. Hence it suffices to prove that every idempotent in the centroid is either zero or the identity.

Let $C$ be such an idempotent. The centroid identities imply that $C$ preserves the lower central series. They also give $[Cz,x]=C[z,x]=0$ for $z\in Z(\nalg)$ and $x\in\nalg$, so $C$ preserves the center. Identify $\nalg/\nalg^2$ with $\bigoplus_i L_i$, where $L_i=\operatorname{span}\{X_i,Y_i,Z_i\}$, and let $\overline C$ be the induced map.

For each $j$, the bracket map $\Lambda^2L_j\to\nalg^2/\nalg^3$ is injective. Indeed, the images of $X_j\wedge Y_j$, $Y_j\wedge Z_j$, and $X_j\wedge Z_j$ are the classes of $P_j,Q_j,v_j$, respectively. These classes are linearly independent because $v_j\neq0$ and $\nalg^2/\nalg^3$ is the direct sum of $W$ and the $P_i,Q_i$ directions.

We now show that $\overline C$ preserves each $L_i$. Fix $x\in L_i$ and $j\neq i$, and let $z_j$ be the $L_j$-component of $\overline Cx$. For every $y\in L_j$, the blocks commute, so $[Cx,y]=C[x,y]=0$. Components of $Cx$ in $\nalg^2$ contribute only to $\nalg^3$, since $[\nalg^2,\nalg]\subset\nalg^3$. Components in the other $L_k$ commute with $y$. Thus $[z_j,y]\in\nalg^3$ for every $y\in L_j$. Injectivity of the bracket map implies $z_j\wedge y=0$ for every $y\in L_j$. Since $\dim L_j=3$, this forces $z_j=0$. Therefore $\overline C(L_i)\subset L_i$.

For $x\in L_i$, we also have $[Cx,x]=C[x,x]=0$. Passing to $\nalg^2/\nalg^3$ and using the same injectivity gives $\overline Cx\wedge x=0$. Hence $\overline Cx$ is a multiple of $x$ for every nonzero $x\in L_i$. Applying this to $X_i,Y_i,Z_i$ and their pairwise sums shows that $\overline C|_{L_i}=a_i\Id_{L_i}$ for some scalar $a_i$. Since $\overline C^2=\overline C$, we have $a_i\in\{0,1\}$.

Let $P$ be the map induced by $C$ on $Z(\nalg)/\nalg^3$, which we identify with $W$. From $v_i=[X_i,Z_i]$ and the centroid identity, we obtain $Cv_i=[CX_i,Z_i]$. Since $CX_i-a_iX_i\in\nalg^2$, it follows that $Cv_i-a_iv_i\in\nalg^3$. Thus $Pv_i=a_iv_i$. Also, $P^2=P$, so $W=\ker P\oplus\im P$. If both values zero and one occurred among the $a_i$, then both summands would be nonzero because the corresponding $v_i$ are nonzero. Every $v_i$ would belong to one of these summands, contradicting linear indecomposability. Therefore all $a_i$ have the same value $a\in\{0,1\}$.

Set $N=C-a\Id$. Then $N$ belongs to the centroid and $N\nalg\subset\nalg^2$. The identity $N[\nalg,\nalg^r]=[\nalg,N\nalg^r]$ gives inductively $N\nalg^r\subset\nalg^{r+1}$. Consequently, $N^3\nalg\subset\nalg^4=0$. Since $C^2=C$, the operator $C$ is diagonalizable, and so is $N=C-a\Id$. A diagonalizable nilpotent endomorphism is zero. Hence $C=a\Id$, which proves that $\nalg$ is indecomposable.

Finally, suppose \eqref{eq:tight-frame} holds. If the configuration splits over $W=W_1\oplus W_2$, let $\Pi$ be the projection onto $W_1$ along $W_2$. Since $\Pi v_i=v_i$ for $v_i\in W_1$ and $\Pi v_i=0$ for $v_i\in W_2$, multiplying \eqref{eq:tight-frame} on the left by $\Pi$ gives $\Pi=(d/m)\sum_{v_i\in W_1}v_i\otimes v_i$. Each operator $v_i\otimes v_i$ is self-adjoint with respect to $\langle\cdot,\cdot\rangle_0$, so $\Pi$ is self-adjoint. Its image and kernel are therefore orthogonal, giving $W_1\perp W_2$. Since the vectors $v_i$ span $W$, both summands contain configuration vectors. No edge joins the two groups of vertices, so the graph is disconnected.

Conversely, suppose the graph is disconnected. Let $W_1$ be the span of the vectors belonging to one connected component, and let $W_2$ be the span of all remaining vectors. Vectors from distinct components have zero inner product, so $W_1\perp W_2$. Both subspaces are nonzero, and the spanning assumption gives $W=W_1\oplus W_2$. This is a decomposition of the configuration, proving the graph criterion.
\end{proof}

The following explicit frames give indecomposable examples in a range of dimensions.

\begin{lemma}\label{lem:Fourier-tight-frames}
For every pair of integers $m>d\geq1$, there is a real unit-norm tight frame of $m$ vectors in $\R^d$ whose graph in Proposition~\ref{prop:configuration-indecomposable} is connected.
\end{lemma}

\begin{proof}
Index the vectors by $j=0,\ldots,m-1$. For $d=2k+1$, set
\begin{equation}\label{eq:odd-Fourier-frame}
v_j=\frac1{\sqrt d}\left(1,\left(\sqrt2\cos\frac{2\ell\pi j}{m},\sqrt2\sin\frac{2\ell\pi j}{m}\right)_{\ell=1}^{k}\right).
\end{equation}
For $d=2k$, set
\begin{equation}\label{eq:even-Fourier-frame}
v_j=\sqrt{\frac2d}\left(\left(\cos\frac{(2\ell-1)\pi j}{m},\sin\frac{(2\ell-1)\pi j}{m}\right)_{\ell=1}^{k}\right).
\end{equation}
The identity $\cos^2 t+\sin^2 t=1$ gives $|v_j|_0^2=(1+2k)/d=1$ in the odd case and $|v_j|_0^2=2k/d=1$ in the even case. When $d=1$, the first formula has no trigonometric coordinates and gives $v_j=1$ for every $j$.

We next verify the tight-frame identity. Let $A$ be the $m\times d$ matrix whose row indexed by $j$ is $v_j^{\mathsf T}$, and set $\theta_j=2\pi j/m$. For every integer $q$ with $0<|q|<m$, the geometric series gives $\sum_{j=0}^{m-1}e^{\mathrm i q\theta_j}=(1-e^{2\pi\mathrm i q})/(1-e^{2\pi\mathrm i q/m})=0$. Thus the corresponding sums of sines and cosines vanish.

In the odd case, products of trigonometric columns involve the integer frequencies $\ell-r$ and $\ell+r$, where $1\leq\ell,r\leq k$. Their nonzero absolute values are at most $2k=d-1<m$. In the even case, the original frequencies are $\ell-1/2$, but products involve the integer frequencies $\ell-r$ and $\ell+r-1$. Their nonzero absolute values are at most $2k-1=d-1<m$. Therefore the same finite-sum identity applies to these products in both cases.

By the product-to-sum identities, the square of an unscaled sine or cosine column has constant term $1/2$, while the product of two distinct such columns has no nonzero constant term. In particular, the product of a sine and a cosine of the same frequency has no constant term. Hence the unscaled trigonometric columns are mutually orthogonal and have squared norm $m/2$. In the odd case, they are also orthogonal to the constant column, since their individual frequencies are nonzero integers less than $m$. The unscaled constant column has squared norm $m$. After including the factors in \eqref{eq:odd-Fourier-frame} and \eqref{eq:even-Fourier-frame}, every column of $A$ has squared norm $m/d$. Thus $A^{\mathsf T}A=(m/d)\Id_{\R^d}$, or equivalently $\sum_{j=0}^{m-1}v_j\otimes v_j=(m/d)\Id_{\R^d}$. Together with $|v_j|_0=1$, this proves \eqref{eq:tight-frame}. In particular, the vectors span $\R^d$.

It remains to prove that the graph is connected. Set $t=\pi/m$ and let $0\leq j<m-1$. Using $\cos\alpha\cos\beta+\sin\alpha\sin\beta=\cos(\alpha-\beta)$, we obtain $d\langle v_j,v_{j+1}\rangle_0=1+2\sum_{\ell=1}^k\cos(2\ell t)$ in the odd case and $d\langle v_j,v_{j+1}\rangle_0=2\sum_{\ell=1}^k\cos((2\ell-1)t)$ in the even case. Multiplying either expression by $\sin t$ and applying $2\sin t\cos\alpha=\sin(\alpha+t)-\sin(\alpha-t)$ makes the sums telescope. The result is $\sin(dt)$ in both cases. Therefore
\begin{equation}\label{eq:Fourier-adjacent-product}
\langle v_j,v_{j+1}\rangle_0=\frac{\sin(d\pi/m)}{d\sin(\pi/m)}>0\qquad(0\leq j<m-1).
\end{equation}
Since $m>d\geq1$, both $\pi/m$ and $d\pi/m$ lie in $(0,\pi)$, which gives the strict inequality. By \eqref{eq:Fourier-adjacent-product}, every pair of consecutive vertices is joined by an edge. Thus the graph contains the path through $0,1,\ldots,m-1$ and is connected.
\end{proof}

\begin{corollary}[Theorem~\ref{thmC}]
For every pair of integers $m>d\geq1$, there is an indecomposable three-step real nilpotent Lie algebra of dimension $6m+d$ which admits a generalized nilsoliton with nonzero harmonic torsion and admits no classical nilsoliton. Such examples also exist in dimension seven and in every dimension $n\geq43$. The smallest dimension of a real nilpotent Lie algebra admitting a generalized nilsoliton but no classical nilsoliton is seven.
\end{corollary}

\begin{proof}
Combine Lemma~\ref{lem:Fourier-tight-frames}, Proposition~\ref{prop:configuration-indecomposable}, Theorem~\ref{thm:tight-frame-construction}, and Proposition~\ref{prop:configuration-not-Einstein}. For $n\geq43$, set $m=\lfloor(n-1)/6\rfloor\geq7$ and $d=n-6m\in\{1,\ldots,6\}$. Then $m>d$ and $n=6m+d$.

For $m=d=1$, the configuration is linearly indecomposable, so Proposition~\ref{prop:single-torsion-seed} gives the seven-dimensional example. Every real nilpotent Lie algebra of dimension at most five admits a classical nilsoliton by \cite{La2002}, and the six-dimensional case is proved in \cite{Wi2003}. The abelian case is included by the flat metric. Hence seven is the smallest possible dimension.
\end{proof}

\section{Intrinsic obstructions to generalized nilsolitons}\label{sec:intrinsic-obstructions}

In this section, we obtain two obstructions to the existence of generalized nilsolitons, using the abelian derived ideal and the quotient by the center. Both apply to every metric and closed three-form, including $H=0$. We also give a condition on third Chevalley--Eilenberg cohomology for nonzero torsion, which is used in Section~\ref{sec:filiform-threshold}.

We first compute the partial trace of the Ricci operator over an abelian derived ideal.

\begin{lemma}\label{lem:metabelian-partial-Ricci}
Let $\nalg$ be a nonabelian metric nilpotent Lie algebra with abelian derived ideal $W=[\nalg,\nalg]$. Let $P_W$ be the orthogonal projection onto $W$, and let $u_1,\ldots,u_s$ be an orthonormal basis of $W^\perp$. Then
\begin{equation}\label{eq:metabelian-partial-Ricci}
\tr(P_W\Ric)=\frac14\sum_{i,j}|[u_i,u_j]|^2>0.
\end{equation}
\end{lemma}

\begin{proof}
Complete the $u_i$ by an orthonormal basis $w_1,\ldots,w_r$ of $W$. Since $\Ric=M_\mu$, all brackets take values in $W$, and $[W,W]=0$, formula \eqref{eq:moment-map-operator} gives
\begin{equation}\label{eq:metabelian-trace-cancellation}
\tr(P_W\Ric)=-\frac12\sum_{a,i}|[w_a,u_i]|^2+\frac14\sum_{i,j}|[u_i,u_j]|^2\\+\frac12\sum_{a,i}|[u_i,w_a]|^2.
\end{equation}
The first and last sums in \eqref{eq:metabelian-trace-cancellation} cancel, proving the equality in \eqref{eq:metabelian-partial-Ricci}.

If the remaining sum were zero, then $[W^\perp,W^\perp]=0$. Together with $[W,W]=0$, this would give $W=[\nalg,\nalg]=[W^\perp,W]=[\nalg,W]$. Every term of the lower central series from the second onward would therefore equal $W$. Nilpotency forces $W=0$, contradicting nonabelianity. Hence the partial trace is strictly positive.
\end{proof}

Combining this partial trace with the soliton trace identity gives a restriction on the eigenvalues of the soliton derivation.

\begin{proposition}[Theorem~\ref{thmD}]\label{prop:metabelian-spectral-obstruction}
Let $\nalg$ be a nonabelian nilpotent Lie algebra whose derived ideal $W$ is abelian of dimension $r$. If $(g,H)$ is a generalized nilsoliton with $c=-\lambda>0$ and symmetric derivation $D$, then
\begin{equation}\label{eq:metabelian-spectral-obstruction}
-3\tr(D^2)+c\bigl(3\tr D+2\tr(D|_W)\bigr)-2rc^2>0.
\end{equation}
In particular, $L=3\tr D+2\tr(D|_W)>0$ and $L^2>24r\tr(D^2)$.
\end{proposition}

\begin{proof}
Let $P_W$ be the $g$-orthogonal projection onto $W$, and set $L=3\tr D+2\tr(D|_W)$. Define $F(t)=-3\tr(D^2)+Lt-2rt^2$ for $t\in\R$. We first prove that $F(c)>0$.

Since $D$ is a derivation, $D[x,y]=[Dx,y]+[x,Dy]\in W$ for all $x,y\in\nalg$. Thus $D$ preserves $W$. Choose an orthonormal basis $w_1,\ldots,w_r$ of $W$ and complete it by an orthonormal basis $u_1,\ldots,u_s$ of $W^\perp$. In this basis, $\tr(P_WD)=\sum_{a=1}^r g(Dw_a,w_a)=\tr(D|_W)$ and $\tr P_W=r$.

Since $D$ is symmetric and $c=-\lambda$, the trace identity \eqref{eq:nilsoliton-trace} gives $\tr(D^2)=c\tr D-c|H|^2/6$. Using this identity, $\tr(H^2)=|H|^2$, and the metric equation $D-c\Id=\Ric-H^2/4$, we obtain
\begin{equation}\label{eq:metabelian-spectral-remainder}
\begin{aligned}
F(c)&=3\bigl(c\tr D-\tr(D^2)\bigr)
+2c\bigl(\tr(D|_W)-rc\bigr)\\
&=\frac c2\tr(H^2)+2c\tr\bigl(P_W(D-c\Id)\bigr)\\
&=2c\tr(P_W\Ric)+\frac c2\tr\bigl((\Id-P_W)H^2\bigr).
\end{aligned}
\end{equation}
By Lemma~\ref{lem:metabelian-partial-Ricci}, $\tr(P_W\Ric)>0$. Since $c>0$, the first term in the last line of \eqref{eq:metabelian-spectral-remainder} is strictly positive. For the second term, the definition of $H^2$ gives $\tr((\Id-P_W)H^2)=\sum_{i=1}^s g(H^2u_i,u_i)=\sum_{i=1}^s|\iota_{u_i}H|^2\geq0$. Therefore $F(c)>0$, proving \eqref{eq:metabelian-spectral-obstruction}. These partial-trace formulas remain valid without assuming that $\Ric$ or $H^2$ preserves $W$.

Finally, $r>0$ because $\nalg$ is nonabelian, and $\tr(D^2)=\|D\|_{\mathrm{HS}}^2\geq0$ because $D$ is symmetric. Hence $F(c)>0$ implies $cL>3\tr(D^2)+2rc^2>0$. As $c>0$, we obtain $L>0$. Completing the square gives $F(c)=-2r(c-L/(4r))^2+(L^2-24r\tr(D^2))/(8r)$. The first term is nonpositive, whereas $F(c)>0$. Thus $L^2-24r\tr(D^2)>0$, which proves the remaining assertion.
\end{proof}

We next use the action of derivations on the quotient by the center. 
\begin{theorem}\label{thm:central-quotient-obstruction}
Let $\nalg$ be a nonabelian nilpotent Lie algebra. If every derivation diagonalizable over $\R$ induces the zero map on $\nalg/Z(\nalg)$, then $\nalg$ admits no generalized nilsoliton.
\end{theorem}

\begin{proof}
Suppose that $(g,H)$ is a generalized nilsoliton with soliton derivation $D$. By Proposition~\ref{prop:nilsoliton-trace}, we have $c=-\lambda>0$. Harmonic reduction gives $\rho(D)H=-2cH$. Set $K=\ker D$ and $V=\im D$. The metric equation makes $D$ symmetric, so it is diagonalizable over $\R$. The hypothesis therefore applies to $D$ and gives $V\subset Z(\nalg)$.

Symmetry also gives $V=K^\perp$, and hence $\nalg=K\oplus V$ orthogonally. Since $Dx$ and $Dy$ are central, the derivation identity gives $D[x,y]=[Dx,y]+[x,Dy]=0$ for all $x,y\in\nalg$. Thus $[\nalg,\nalg]\subset K$, so $K$ is an ideal. Since $V$ is central, both subspaces are ideals and $[K,V]=0$. Consequently, this decomposition is an orthogonal Lie direct sum, with $[\nalg,\nalg]=[K,K]\neq0$.

We next determine the metric equation on $V$. For $v\in V$, all brackets $[v,x]$ vanish, and every bracket $[x,y]$ belongs to $K=V^\perp$. Both contractions in \eqref{eq:moment-map-operator} therefore vanish whenever one argument is $v$. Hence $\Ric v=0$. The metric equation now gives $H^2v=4(cv-Dv)\in V$, so $H^2$ preserves $V$ and
\begin{equation}\label{eq:central-quotient-D-bound}
D|_V=c\Id_V-\frac14H^2|_V\leq c\Id_V.
\end{equation}
Here the inequality follows from $g(H^2v,v)=|\iota_vH|^2\geq0$. If $Dv=cv$, then the same identity gives $|\iota_vH|^2=0$, and therefore $\iota_vH=0$.

Choose an orthonormal eigenbasis $\{e_i\}$ of $D$ adapted to $K\oplus V$, and write $De_i=\delta_i e_i$. The eigenvalues on $K$ are zero. Together with $c>0$ and \eqref{eq:central-quotient-D-bound}, this shows that every eigenvalue of $D$ is at most $c$. Evaluating $\rho(D)H=-2cH$ on $e_i,e_j,e_k$ gives $(\delta_i+\delta_j+\delta_k-2c)H(e_i,e_j,e_k)=0$.

Suppose that $e_i\in K$ and $H(e_i,e_j,e_k)\neq0$. Since $\delta_i=0$, the preceding identity gives $\delta_j+\delta_k=2c$. Both eigenvalues are at most $c$, so $\delta_j=\delta_k=c$. As $c\neq0$, the corresponding eigenvectors belong to $V$. By the equality case considered above, $\iota_{e_j}H=\iota_{e_k}H=0$, contradicting $H(e_i,e_j,e_k)\neq0$. Thus every component with an argument in $K$ vanishes. By linearity and alternation, $\iota_xH=0$ for every $x\in K$. Equivalently, $H$ is the extension by zero of a three-form on $V$, and $H^2|_K=0$.

Since the decomposition is an orthogonal Lie direct sum, $\Ric|_K=\Ric_K$. The metric equation on $K$ therefore reduces to $\Ric_K=-c\Id_K$. However, $K$ is nilpotent and $[K,K]\neq0$, so its center contains a nonzero vector $z$. For an orthonormal basis $\{f_a\}$ of $K$, centrality of $z$ and \eqref{eq:moment-map-operator} give $g(\Ric_Kz,z)=\frac14\sum_{a,b}g([f_a,f_b],z)^2\geq0$. On the other hand, $\Ric_K=-c\Id_K$ gives $g(\Ric_Kz,z)=-c|z|^2<0$. This contradiction proves the theorem.
\end{proof}

The nonabelian hypothesis in Theorem~\ref{thm:central-quotient-obstruction} is necessary. On $\R^3$ with an orthonormal coframe, $H=e^{123}$, $\lambda=-3/2$, and $D=\Id$ satisfy \eqref{eq:reduced-soliton-system}, although the quotient by the center is zero.

Recall that a Lie algebra is characteristically nilpotent if all its derivations are nilpotent. A nonabelian characteristically nilpotent Lie algebra admits no generalized nilsoliton by the positive-trace requirement in Proposition~\ref{prop:nilsoliton-trace}, as already observed in \cite[Section~7, after Question~7.14]{FuLaSt2024}. Burde's construction gives examples in every dimension $n\geq7$ \cite[Proposition~2.13]{Bu2002}. We show that this nonexistence persists after adjoining an abelian factor.

\begin{corollary}\label{cor:abelian-stabilization-obstruction}
Let $\mathfrak c$ be a nonabelian characteristically nilpotent real Lie algebra. Then $\mathfrak c\oplus\R^k$ admits no generalized nilsoliton for any $k\geq0$.
\end{corollary}

\begin{proof}
Let $A_0$ be an arbitrary derivation of $\mathfrak c\oplus\R^k$, and write its matrix relative to this decomposition as $A_0=\left(\begin{smallmatrix}A&B\\ C&F\end{smallmatrix}\right)$. Applying the derivation identity to two vectors in $\mathfrak c$ gives $A\in\Der(\mathfrak c)$ and $C([\mathfrak c,\mathfrak c])=0$. Applying it to one vector in $\mathfrak c$ and one in $\R^k$ gives $\im B\subset Z(\mathfrak c)$.

Since $Z(\mathfrak c\oplus\R^k)=Z(\mathfrak c)\oplus\R^k$, the map induced by $A_0$ on the quotient by the center is the map induced by $A$ on $\mathfrak c/Z(\mathfrak c)$. It is nilpotent because $A$ is nilpotent. If $A_0$ is diagonalizable over $\R$, its induced map on the quotient is also diagonalizable and must therefore be zero. Theorem~\ref{thm:central-quotient-obstruction} proves the claim.
\end{proof}

For $k\geq1$, the derivation $0\oplus\Id_{\R^k}$ is diagonalizable and has trace $k>0$. Thus the positive-trace requirement alone does not exclude the Lie algebras in Corollary~\ref{cor:abelian-stabilization-obstruction}.

Finally, nonzero torsion gives a nonzero eigenvector in third Chevalley--Eilenberg cohomology. Here $H^3_{\mathrm{CE}}(\nalg)$ denotes the cohomology of the invariant complex.

\begin{proposition}\label{prop:CE-spectral-prerequisite}
Let $(g,H)$ be a generalized nilsoliton on a nilpotent Lie algebra $\nalg$, with $H\neq0$, $c=-\lambda>0$, and derivation $D$. Then $[H]\neq0$ in $H^3_{\mathrm{CE}}(\nalg)$ and
\begin{equation}\label{eq:CE-spectral-prerequisite}
\rho(D)[H]=-2c[H].
\end{equation}
\end{proposition}

\begin{proof}
Since $D$ is a derivation, $\rho(D)$ commutes with $d_\mu$ and therefore acts on Chevalley--Eilenberg cohomology. Harmonicity gives $d_\mu H=\CEadj H=0$. If $H=d_\mu\alpha$ for an invariant two-form $\alpha$, then \eqref{eq:CE-adjoint-pairing} gives $|H|^2=3\bar g(\CEadj H,\alpha)=0$, a contradiction. Hence $[H]\neq0$. The torsion equation $\rho(D)H=-2cH$ descends to this class and proves \eqref{eq:CE-spectral-prerequisite}.
\end{proof}

\section{The filiform family \texorpdfstring{$\mathfrak m_2(n)$}{m2(n)}}\label{sec:filiform-threshold}

In this section, we prove that $\mathfrak m_2(n)$ admits a generalized nilsoliton if and only if $5\leq n\leq8$. We construct metrics with nonzero harmonic torsion in these dimensions. The spectral obstruction from Section~\ref{sec:intrinsic-obstructions} excludes $n\geq10$. In dimension nine, we combine it with a computation of third Chevalley--Eilenberg cohomology to exclude every metric and closed three-form.

Fix $n\geq5$ and the basis $E_1,\ldots,E_n$ in \eqref{eq:m2-introduction}. The only potentially nonzero Jacobi identities involve $E_1,E_2,E_j$, and their two terms cancel. The derived ideal $W=\operatorname{span}\{E_3,\ldots,E_n\}$ is abelian. The lower central terms are $\gamma_r=\operatorname{span}\{E_{r+1},\ldots,E_n\}$ for $2\leq r\leq n-1$, followed by $\gamma_n=0$, so the algebra is filiform. Its center is $\R E_n$, which also proves indecomposability, since a direct sum of two nonzero nilpotent Lie algebras has center dimension at least two.

The map $\phi E_i=iE_i$ is a positive derivation. Assign weight $i$ to $E^i$ and weight $i_1+\cdots+i_p$ to $E^{i_1}\wedge\cdots\wedge E^{i_p}$. The differential preserves weight. Write $H^p_{\mathrm{CE},q}$ for the weight-$q$ part of Chevalley--Eilenberg cohomology.

\begin{lemma}\label{lem:m2-arbitrary-derivations}
Every derivation $A$ of $\mathfrak m_2(n)$ is lower triangular in the fixed basis, with diagonal $t(1,2,\ldots,n)$ for some $t\in\R$. If $b_{3,q}=\dim H^3_{\mathrm{CE},q}(\mathfrak m_2(n))$, the action induced by $\rho(A)$ on $H^3_{\mathrm{CE}}(\mathfrak m_2(n))$ has characteristic polynomial $\prod_q(\xi+tq)^{b_{3,q}}$.
\end{lemma}

\begin{proof}
Every derivation preserves the lower central series, so $AE_i\in\operatorname{span}\{E_i,\ldots,E_n\}$ for $i\geq3$. In the derivation identity for $[E_2,E_3]=E_5$, the $E_4$ coefficient on the right is the $E_1$ coefficient of $AE_2$. The left side has no $E_4$ component, so this coefficient vanishes and $A$ is lower triangular.

Let $a,b$ be its first two diagonal entries. Recursion through $E_i=[E_1,E_{i-1}]$ gives the $i$th diagonal entry $(i-2)a+b$ for $i\geq2$. Comparing the $E_5$ coefficients in the identity for $[E_2,E_3]=E_5$ gives $b+(a+b)=3a+b$, and hence $b=2a$. Taking $t=a$ proves the first assertion.

Write $A=t\phi+N$, where $N$ is a strictly lower triangular derivation. The action $\rho(N)$ commutes with $d_\mu$ and replaces a covector $E^j$ only by covectors $E^i$ with $i<j$. It therefore strictly lowers cochain weight and induces a weight-lowering map on cohomology. On $H^3_{\mathrm{CE},q}$, the diagonal part $t\rho(\phi)$ acts by $-tq$. Thus the induced action is triangular with the asserted characteristic polynomial.
\end{proof}

\begin{proposition}\label{prop:m2-classical-range}
The algebra $\mathfrak m_2(n)$ admits a classical nilsoliton metric exactly when $5\leq n\leq6$.
\end{proposition}

\begin{proof}
Existence in dimensions five and six follows from \cite{La2002,Wi2003}. Nonexistence is proved in \cite[Theorem~32]{Pa2010} for $n=7$ and in \cite[Theorem~2(1) and Section~4.1]{Ni2008} for $n\geq8$. We give a direct proof using the partial Ricci trace.

Suppose $\Ric=-c\Id+D$, with $c>0$. By Lemma~\ref{lem:m2-arbitrary-derivations}, the eigenvalues of $D$ are $t,2t,\ldots,nt$, and positive trace gives $t>0$. Put $T=\sum_{i=1}^n i$ and $Q=\sum_{i=1}^n i^2$. The identity $\tr(\Ric D)=0$ gives $-ctT+t^2Q=0$, so $t/c=T/Q=3/(2n+1)$. If $P_W$ is orthogonal projection onto the derived ideal, then
\begin{equation}\label{eq:m2-classical-compressed-trace}
\tr(P_W\Ric)=-c(n-2)+t(T-3)
=\frac{c(n-2)(7-n)}{2(2n+1)}\leq0\qquad(n\geq7).
\end{equation}
Equation~\eqref{eq:m2-classical-compressed-trace} contradicts Lemma~\ref{lem:metabelian-partial-Ricci}.
\end{proof}

We now construct generalized nilsolitons with nonzero torsion. In each construction, the fixed basis is orthogonal, $g_i=g(E_i,E_i)$, and $e_i=E_i/\sqrt{g_i}$, so $e^i=\sqrt{g_i}E^i$. A bracket $[E_i,E_j]=E_k$ has orthonormal coefficient $\sqrt{g_k/(g_ig_j)}$. Thus the coefficients below are determined by metrics on the fixed Lie algebra.

The fixed basis is nice, and the torsion triples used below have pairwise intersections of size at most one. Hence $\Ric$ and $H^2$ are diagonal by Section~\ref{sec:weights}. Also, distinct cochain weights are orthogonal for a diagonal metric. Since $d_\mu$ preserves weight, so does $\CEadj$. These observations reduce the tensor and harmonicity calculations to the components displayed below.

\begin{proposition}\label{prop:m2-low-dimensional-existence}
The following metrics and three-forms define generalized nilsolitons with nonzero torsion on $\mathfrak m_2(5)$ and $\mathfrak m_2(6)$.
\begin{equation}\label{eq:m2-five-data}
n=5:\quad g=\diag(1,2,60/7,30,360/7),\quad H=\sqrt{5/7}\,e^{125}-\frac5{\sqrt{14}}e^{134},\quad \lambda=-8,\quad D=2\phi;
\end{equation}
\begin{equation}\label{eq:m2-six-data}
n=6:\quad g=\diag(1,3,16,96,256,768),\quad H=\sqrt3\,e^{156}+\frac{2\sqrt6}{3}e^{246}-\sqrt6\,e^{345},\quad \lambda=-12,\quad D=2\phi.
\end{equation}
\end{proposition}

\begin{proof}
For \eqref{eq:m2-five-data}, order the brackets as $12\to3$, $13\to4$, $14\to5$, $23\to5$. Their coefficient squares are $(30/7,7/2,12/7,3)$. Both torsion summands are closed and have weight eight. The weight-eight two-form space is $\R e^{35}$, and $de^{35}=-\sqrt{30/7}\,e^{125}-\sqrt{12/7}\,e^{134}$. Its exterior-inner-product pairing with $H$ is zero because $\sqrt{30/7}\sqrt{5/7}=\sqrt{12/7}(5/\sqrt{14})=5\sqrt6/7$. Hence $H$ is harmonic. The full tensors are
\begin{equation}\label{eq:m2-five-tensors}
\Ric=\diag\left(-\frac{19}4,-\frac{51}{14},-\frac{31}{28},\frac{25}{28},\frac{33}{14}\right),\quad H^2=\diag\left(5,\frac{10}7,\frac{25}7,\frac{25}7,\frac{10}7\right).
\end{equation}
Equation~\eqref{eq:m2-five-tensors} gives $\Ric-H^2/4=\diag(-6,-4,-2,0,2)=-8\Id+2\phi$. The derivation $D=2\phi$ is symmetric, and weight eight gives $\rho(D)H=-16H=2\lambda H$.

For \eqref{eq:m2-six-data}, order the brackets as $12\to3$, $13\to4$, $14\to5$, $15\to6$, $23\to5$, $24\to6$, and write their positive coefficients as $a_1,\ldots,a_6$. The metric gives squares $(16/3,6,8/3,3,16/3,8/3)$. For $H=h_1e^{156}+h_2e^{246}+h_3e^{345}$, we have
\begin{equation}\label{eq:m2-six-closedness}
dH=(-a_2h_2+a_5h_1)e^{1236}
+(-a_1h_3-a_4h_2-a_6h_1)e^{1245}.
\end{equation}
Substitution into \eqref{eq:m2-six-closedness} gives the coefficients $-4+4$ and $4\sqrt2-2\sqrt2-2\sqrt2$, both zero. The form has weight twelve, whereas every two-form has weight at most eleven. Since $\CEadj$ preserves weight, $\CEadj H=0$. The full tensors are
\begin{equation}\label{eq:m2-six-tensors}
\Ric=\diag\left(-\frac{17}2,-\frac{20}3,-3,\frac13,\frac52,\frac{17}6\right),\quad H^2=\diag\left(6,\frac{16}3,12,\frac{52}3,18,\frac{34}3\right).
\end{equation}
Equation~\eqref{eq:m2-six-tensors} gives $\Ric-H^2/4=\diag(-10,-8,-6,-4,-2,0)=-12\Id+2\phi$. Again $D=2\phi$ is symmetric, and $\rho(D)H=-24H=2\lambda H$. Thus both constructions satisfy \eqref{eq:reduced-soliton-system}.
\end{proof}

In dimension seven, the following parameters give a solution by radicals.

\begin{proposition}\label{prop:m2-seven-existence}
Let $s=\sqrt{37}$ and set
\begin{equation}\label{eq:m2-seven-parameters}
A=\frac2{11},\quad B=\frac{74-6s}{11},\quad C=\frac{6s}{11},\quad F=\frac{48}{11},\quad P=\frac{6(s-3)}{11},\quad Q=\frac{92-12s}{11}.
\end{equation}
On $\mathfrak m_2(7)$ take
\begin{equation}\label{eq:m2-seven-data}
g=\diag(1,C,AC,ABC,ABC^2,ABC^2F,ABC^3F), H=\sqrt P(e^{137}+e^{236})-\sqrt Q\,e^{146},\quad \lambda=-11,\quad D=2\phi.
\end{equation}
These data define a generalized nilsoliton with nonzero torsion.
\end{proposition}

\begin{proof}
All six parameters are positive because $6<s<7$, and direct multiplication gives $BP=CQ$. Set $a=\sqrt A$, $b=\sqrt B$, $c_1=\sqrt C$, and $f=\sqrt F$. In the bracket order $12\to3$, $13\to4$, $14\to5$, $15\to6$, $16\to7$, $23\to5$, $24\to6$, $25\to7$, the metric in \eqref{eq:m2-seven-data} gives coefficients $(a,b,c_1,f,c_1,b,f,f)$.

Put $h=\sqrt P$. Since $BP=CQ$, the torsion can be written as $H=h e^{137}-(b/c_1)h e^{146}+h e^{236}$. For a form $h_1e^{137}+h_2e^{146}+h_3e^{236}$ and the displayed bracket coefficients, the differential and codifferential are
\begin{equation}\label{eq:m2-seven-harmonic-equations}
dH=f(h_1-h_3)e^{1235}, \quad\CEadj H=-(bh_1+c_1h_2)e^{47}-(c_1h_2+bh_3)e^{56}.
\end{equation}
The weight-eleven four-form space is $\R e^{1235}$, and the weight-eleven two-form space is spanned by $e^{47},e^{56}$, so no other components occur. Substituting $(h_1,h_2,h_3)=(h,-bh/c_1,h)$ into \eqref{eq:m2-seven-harmonic-equations} proves harmonicity. The full diagonal tensors are
\begin{equation}\label{eq:m2-seven-tensors}
\begin{aligned}
2\Ric={}&\diag(-A-B-2C-F,-A-B-2F,A-2B, B-C-F,B+C-2F,2F-C,C+F),\\
H^2={}&2\diag(P+Q,P,2P,Q,0,P+Q,P).
\end{aligned}
\end{equation}
Substitution of \eqref{eq:m2-seven-parameters} into \eqref{eq:m2-seven-tensors} gives $\Ric-H^2/4=\diag(-9,-7,-5,-3,-1,1,3)=-11\Id+2\phi$. The map $D=2\phi$ is a symmetric derivation, and weight eleven gives $\rho(D)H=-22H=2\lambda H$. This proves \eqref{eq:reduced-soliton-system}.
\end{proof}

We next construct a solution in dimension eight. Its parameters are rational functions of a root of an explicit polynomial. In this construction, $a,v,w,z,r,U,V,W,t$ are scalars. Define
\begin{equation}\label{eq:eight-root-polynomials}
\begin{aligned}
\mathscr P(y)={}&90326016y^7-51492672y^6+138824912y^5\\
&-106249414y^4-78228073y^3-48787989y^2\\
&-42054372y+26823744,\\
\mathscr C(y)={}&69984y^5-857520y^4-606034y^3\\
&+347995y^2-304056y-415872,\\
\mathscr N(y)={}&177120y^5+339640y^4+256060y^3\\
&+37840y^2-123150y-54720.
\end{aligned}
\end{equation}
Set $y_-=11395227/10^7$, $y_+=11395228/10^7$, and $I=[y_-,y_+]$. For $y$ with $\mathscr C(y)(y+1)\neq0$, define
\begin{equation}\label{eq:eight-root-rational-functions}
t=-\frac{\mathscr N(y)}{\mathscr C(y)},\quad z=\frac{67t/10-1}{y+1},\quad v=yz, \quad W=2v-\frac{13}{5}t,\quad V=yW.
\end{equation}
Define the remaining parameters by the affine expressions below. We also record equivalent expressions for $v,z$, which are used to verify the metric equation.
\begin{equation}\label{eq:m2-eight-linear-parameters}
\begin{aligned}
a&=2V+W-\frac{32}{5}t,&v&=\frac W2+\frac{13}{10}t,\\
w&=-V-\frac W2+\frac{83}{10}t-1,&z&=-\frac W2+\frac{27}{5}t-1,\\
r&=2-2t,&U&=-V-\frac W2+\frac{22}{5}t.
\end{aligned}
\end{equation}
The expressions for $v,z$ agree with \eqref{eq:eight-root-rational-functions} because $W=2yz-13t/5$ and $(1+y)z=67t/10-1$.

\begin{lemma}\label{lem:m2-eight-parameter-certificate}
The polynomial $\mathscr P$ has a unique zero $y_*\in I$. Evaluated at $y_*$, all nine parameters $a,v,w,z,r,U,V,W,t$ belong to $(1/3,3/2)$ and satisfy
\begin{equation}\label{eq:m2-eight-compatibility}
wz=r,\qquad Va=Uw,\qquad Wav=Ur.
\end{equation}
\end{lemma}

\begin{proof}
Using \eqref{eq:eight-root-polynomials}, exact endpoint evaluation and interval Horner evaluation of $\mathscr P'$ give
\begin{equation}\label{eq:eight-root-isolation}
-28<\mathscr P(y_-)<-27,\qquad 60<\mathscr P(y_+)<61,\quad \mathscr P'(y)>874000000\qquad(y\in I).
\end{equation}
Thus \eqref{eq:eight-root-isolation} and the intermediate value theorem give a unique root $y_*\in(y_-,y_+)$. Interval Horner evaluation also gives
\begin{equation}\label{eq:eight-denominator-bounds}
-2518642<\mathscr C(y)<-2518640,\qquad
1145965<\mathscr N(y)<1145967\qquad(y\in I).
\end{equation}
Since $y+1>2$, equation~\eqref{eq:eight-denominator-bounds} ensures that all rational functions are defined on $I$.

To obtain Table~\ref{tab:m2-eight-parameter-bounds}, retain the exact rational endpoints from the Horner evaluations of $\mathscr C$ and $\mathscr N$ and apply interval operations in the order $t,z,v,W,V,a,w,r,U$. Products use the minimum and maximum of the four endpoint products, and division uses the reciprocal interval. The table gives the asserted positivity bounds throughout $I$.

\begin{table}[htbp]
\centering
\begin{tabular}{crrrrrrrrr}
\hline
&$a$&$v$&$w$&$z$&$r$&$U$&$V$&$W$&$t$\\
\hline
Lower&3639&10910&11384&9574&10900&3639&11384&9990&4549\\
Upper&3641&10911&11385&9575&10901&3640&11385&9991&4550\\
\hline
\end{tabular}
\caption{Strict rational bounds for $10^4$ times each parameter on $I$.}\label{tab:m2-eight-parameter-bounds}
\end{table}

It remains to prove the compatibility equations at $y_*$. Substituting \eqref{eq:eight-root-rational-functions} and \eqref{eq:m2-eight-linear-parameters}, clearing denominators, and expanding gives
\begin{equation}\label{eq:eight-exact-factor-identities}
\begin{aligned}
\mathscr C(y)^2(wz-r)
&=-134(2y+1)(108y^2-55y-96)\mathscr P(y),\\
\mathscr C(y)^2(Va-Uw)
&=6(2y+1)(3888y^4+396y^3+2044y^2-777y-1824)\mathscr P(y).
\end{aligned}
\end{equation}
Since $\mathscr P(y_*)=0$ and $\mathscr C(y_*)\neq0$, equation~\eqref{eq:eight-exact-factor-identities} gives $wz=r$ and $Va=Uw$. The definitions $V=yW$ and $v=yz$ give $Wav-Ur=z(Va-Uw)+U(wz-r)$, proving the third equation in \eqref{eq:m2-eight-compatibility}.
\end{proof}

\begin{proposition}\label{prop:m2-eight-existence}
Evaluate the parameters at the root $y_*$ in Lemma~\ref{lem:m2-eight-parameter-certificate}. The following metric and three-form on $\mathfrak m_2(8)$ define a generalized nilsoliton with nonzero torsion.
\begin{equation}\label{eq:m2-eight-data}
\begin{aligned}
g&=\diag(1,1,a,av,av,avw,avwz,avwz), \quad H=B_0(E^{168}+E^{267}+E^{348}-E^{357}+E^{456}),\\
B_0^2&=U(avw)(avwz),\quad B_0>0,\quad
\lambda=-\frac{15}{2}t,\quad D=t\phi.
\end{aligned}
\end{equation}
Here the form is expressed in the original fixed coframe $E^i$.
\end{proposition}

\begin{proof}
Lemma~\ref{lem:m2-eight-parameter-certificate} makes every metric entry and $B_0^2$ positive. Order the brackets as $12\to3$, $13\to4$, $14\to5$, $15\to6$, $16\to7$, $17\to8$, $23\to5$, $24\to6$, $25\to7$, $26\to8$. Computing $g_k/(g_ig_j)$ from \eqref{eq:m2-eight-data} gives the coefficient squares
\begin{equation}\label{eq:m2-eight-actual-squares}
(a,v,1,w,z,1,v,w,r,z).
\end{equation}
The ninth entry is $wz=r$ by \eqref{eq:m2-eight-compatibility}. For the ordered torsion triples $168,267,348,357,456$, the orthonormal coefficient squares are $(U,U,V,V,W)$, with the fourth coefficient negative. The first two squares equal $U$ by the definition of $B_0^2$. The next two are $Uw/a=V$, and the last is $Uwz/(av)=Ur/(av)=W$. Hence the metric and fixed three-form realize all the stated coefficients.

For closedness, the structure equations give
\begin{equation}\label{eq:m2-eight-closedness}
\begin{aligned}
dE^{168}&=E^{1248}, dE^{267}=-E^{1257}, \quad dE^{348}=-E^{1248}-E^{1347}-E^{2346},\\
dE^{357}&=-E^{1257}-E^{1347}-E^{1356}, \quad dE^{456}=-E^{1356}+E^{2346}.
\end{aligned}
\end{equation}
Substituting the signs in \eqref{eq:m2-eight-data} into \eqref{eq:m2-eight-closedness} gives $dH=0$. The form has weight fifteen, and the weight-fifteen two-form space is $\R E^{78}$. Since $dE^{78}=-E^{168}-E^{258}+E^{267}$, we obtain
\begin{equation}\label{eq:m2-eight-coclosedness}
\langle H,dE^{78}\rangle_\wedge
=B_0\left(-\frac1{g_1g_6g_8}+\frac1{g_2g_6g_7}\right)=0,
\end{equation}
because $g_1=g_2=1$ and $g_7=g_8$. Equation~\eqref{eq:m2-eight-coclosedness} and orthogonality of distinct weights give $\CEadj H=0$. Thus $H$ is harmonic.

It remains to verify the metric equation. The full torsion contraction is
\begin{equation}\label{eq:m2-eight-H-square}
H^2=2\diag(U,U,2V,V+W,V+W,2U+W,U+V,U+V).
\end{equation}
Using \eqref{eq:m2-eight-actual-squares} and \eqref{eq:m2-eight-H-square}, the eight entries of $2R=2(\Ric-H^2/4)$ are
\begin{equation}\label{eq:m2-eight-full-R}
\begin{aligned}
2R_{11}&=-U-a-v-w-z-2,&2R_{22}&=-U-a-r-v-w-z,\\
2R_{33}&=-2V+a-2v,&2R_{44}&=-V-W+v-w-1,\\
2R_{55}&=-V-W-r+v-w+1,&2R_{66}&=-2U-W+2w-2z,\\
2R_{77}&=-U-V+r+z-1,&2R_{88}&=-U-V+z+1.
\end{aligned}
\end{equation}
Substituting \eqref{eq:m2-eight-linear-parameters} into \eqref{eq:m2-eight-full-R} gives
\begin{equation}\label{eq:m2-eight-soliton-R}
2R=t\diag(-13,-11,-9,-7,-5,-3,-1,1),\qquad R=-\frac{15}{2}t\Id+t\phi.
\end{equation}
Equation~\eqref{eq:m2-eight-soliton-R} proves the metric equation. The map $D=t\phi$ is a symmetric derivation, and weight fifteen gives $\rho(D)H=-15tH=2\lambda H$. Hence \eqref{eq:reduced-soliton-system} holds.
\end{proof}

We now exclude generalized nilsolitons for $n\geq9$. Suppose that a soliton exists, and set $c=-\lambda>0$. Write $W=[\mathfrak m_2(n),\mathfrak m_2(n)]$. By Lemma~\ref{lem:m2-arbitrary-derivations}, the eigenvalues of $D$ on the whole algebra and on $W$ are $t(1,\ldots,n)$ and $t(3,\ldots,n)$, respectively. Positive trace gives $t>0$. Put
\begin{equation}\label{eq:m2-spectral-data}
T=\frac{n(n+1)}2,\qquad S=T-3,\qquad
Q=\frac{n(n+1)(2n+1)}6,\qquad r=n-2.
\end{equation}
If $n\geq10$, Proposition~\ref{prop:metabelian-spectral-obstruction} requires $\Delta_n=(5T-6)^2-24(n-2)Q>0$. However,
\begin{equation}\label{eq:m2-discriminant}
4\Delta_n=n^3(66-7n)-15n^2-88n+144<0\qquad(n\geq10).
\end{equation}
Indeed, $66-7n\leq-4$, and the remaining quadratic expression is negative throughout this range. Thus every metric and every torsion form are excluded for $n\geq10$.

\begin{lemma}\label{lem:m2-nine-cohomology}
The weight-sixteen component of $H^3_{\mathrm{CE}}(\mathfrak m_2(9))$ is zero.
\end{lemma}

\begin{proof}
The weight-sixteen three-form basis, in order, is $E^{169},E^{178},E^{259},E^{268},E^{349},E^{358},E^{367},E^{457}$. Write their coefficients as $a_1,\ldots,a_8$. The structure equations $dE^3=-E^{12}$, $dE^4=-E^{13}$, and $dE^k=-E^{1,k-1}-E^{2,k-2}$ for $5\leq k\leq9$ give the complete closedness system
\begin{equation}\label{eq:m2-nine-CE-system}
\begin{gathered}
a_1-a_3-a_5=0,\quad a_2-a_3-a_4-a_6=0,\quad a_1-a_2-a_4-a_7=0,\quad -a_5-a_6=0,\\
-a_6-a_7-a_8=0,\quad -a_8=0, \quad -a_5-a_7+a_8=0,\quad -a_6+a_7=0.
\end{gathered}
\end{equation}
These equations are the coefficients in the weight-sixteen four-form basis $E^{1249}$, $E^{1258}$, $E^{1267}$, $E^{1348}$, $E^{1357}$, $E^{1456}$, $E^{2347}$, and $E^{2356}$. The last five equations of \eqref{eq:m2-nine-CE-system} give $a_5=a_6=a_7=a_8=0$. The first three then give $a_4=0$ and $a_1=a_2=a_3$. Thus the closed space is $\R(E^{169}+E^{178}+E^{259})$. Since $dE^{79}=-(E^{169}+E^{178}+E^{259})$, every closed form of weight sixteen is exact.
\end{proof}

\begin{proposition}\label{prop:m2-nine-nonexistence}
The algebra $\mathfrak m_2(9)$ admits no generalized nilsoliton.
\end{proposition}

\begin{proof}
Proposition~\ref{prop:m2-classical-range} excludes $H=0$. If $H\neq0$, Proposition~\ref{prop:CE-spectral-prerequisite} gives a nonzero cohomology class with eigenvalue $-2c$ under $\rho(D)$. By Lemma~\ref{lem:m2-arbitrary-derivations}, this eigenvalue must equal $-tq$ for an integer $q$ with $H^3_{\mathrm{CE},q}\neq0$. Hence $q=2c/t$. For $n=9$, equation~\eqref{eq:m2-spectral-data} gives $T=45$, $S=42$, $Q=285$, and $r=7$. Substituting $c=tq/2$ into \eqref{eq:metabelian-spectral-obstruction} and multiplying by $2/t^2$, we obtain
\begin{equation}\label{eq:m2-nine-spectral-window}
-7q^2+219q-1710=-(q-15)(7q-114)>0.
\end{equation}
Equation~\eqref{eq:m2-nine-spectral-window} gives $15<q<114/7$, whose only integer is sixteen. This contradicts Lemma~\ref{lem:m2-nine-cohomology}. The argument uses the action of an arbitrary derivation on cohomology. It does not assume that the metric is diagonal or that $H$ has a single weight in the fixed basis, so it excludes all metrics and closed three-forms.
\end{proof}

\begin{theorem}[Theorem~\ref{thm:filiform-threshold-introduction}]\label{thm:filiform-threshold}
For every integer $n\geq5$, the Lie algebra $\mathfrak m_2(n)$ admits a generalized nilsoliton if and only if $5\leq n\leq8$. In each of these dimensions, there exists a generalized nilsoliton with nonzero torsion.
\end{theorem}

\begin{proof}
Proposition~\ref{prop:m2-low-dimensional-existence} gives generalized nilsolitons with nonzero torsion in dimensions five and six. Propositions~\ref{prop:m2-seven-existence} and~\ref{prop:m2-eight-existence} give such solutions in dimensions seven and eight, respectively. Hence $\mathfrak m_2(n)$ admits a generalized nilsoliton with nonzero torsion whenever $5\leq n\leq8$.

For $n\geq10$, \eqref{eq:m2-discriminant} excludes every generalized nilsoliton. Proposition~\ref{prop:m2-nine-nonexistence} proves the same nonexistence in dimension nine. Thus $\mathfrak m_2(n)$ admits no generalized nilsoliton for $n\geq9$, completing the proof.
\end{proof}

\section*{Acknowledgements} 
This work was supported by the National Natural Science Foundation of China (12571025 and 12131012) and the Natural Science Research of Jiangsu Education Institutions of China (No. 23KJB110016).


\end{document}